\documentclass[12pt]{amsart}

\usepackage{amsmath, amsthm, mathrsfs,amssymb, amsfonts}
\usepackage[noadjust]{cite}

\usepackage{graphicx}
\usepackage[dvipsnames]{xcolor}

\usepackage{mathtools}
\usepackage{cmap}

\usepackage{dsfont}
\usepackage{multicol}
\usepackage{multirow}

\usepackage{url}

\usepackage[bookmarks=true,
    colorlinks=true,
    linkcolor=purple,
    citecolor=red,
    filecolor=orange,
    menucolor=purple
    urlcolor=green,
    breaklinks=true]{hyperref}

\usepackage{fullpage}

\usepackage[all]{xy}

\usepackage{bbm}
\newcommand{\N}{{\mathds N}}   
\newcommand{\R}{{\mathds R}}   

\newcommand{\id}{\mathrm{id}} 
\newcommand{\proj}{\mathrm{proj}} 
\newcommand{\conv}{\mathrm{conv}} 
\newcommand{\basis}{\varepsilon} 
\newcommand{\Hgr}{\mathrm{H}} 

\newcommand{\vtc}{\mathrm{vert}} 

\newcommand{\Dowker}{\mathrm{Dowker}} 
\newcommand{\dowker}{\mathcal{D}} 

\newcommand{\rect}{\mathrm{Rect}} 
\newcommand{\cuboid}{\mathrm{Cuboid}} 
\newcommand{\simp}{\mathrm{simp}} 

\numberwithin{equation}{subsection}

\newtheorem{theorem}[equation]{Theorem}

\newtheorem{proposition}[equation]{Proposition}
\newtheorem{lemma}[equation]{Lemma}
\newtheorem{corollary}[equation]{Corollary}

\theoremstyle{definition}
\newtheorem{example}[equation]{Example}

\newtheorem{definition}[equation]{Definition}
\newtheorem*{definition*}{Definition}
\newtheorem{def-prop}[equation]{Definition--Proposition}

\newtheorem{convention}[equation]{Convention}

\newtheorem{remark}[equation]{Remark}
\newtheorem*{remark*}{Remark}

\usepackage[colorinlistoftodos]{todonotes}

\title{Dowker's Theorem for Higher-Order Relations}
\author{Vin de Silva, Chad Giusti, Vladimir Itskov, Michael Robinson, \\ Radmila Sazdanovic, Niko Schonsheck,  Melvin Vaupel,  Iris Yoon}
\thanks{The authors are listed in alphabetical order. }
\date{\today}

\begin{document}

\begin{abstract}
Given a relation $R \subseteq I \times J$ between two sets, Dowker's Theorem (1952) states that the homology groups of two associated simplicial complexes---now known as Dowker complexes---are isomorphic. In its modern form, the full result asserts a functorial homotopy equivalence between the two Dowker complexes.
What can be said about relations defined on three or more sets?
We present a simple generalization to `multiway' relations of the form $R \subseteq I_1 \times I_2 \times \cdots \times I_m$. The theorem asserts functorial homotopy equivalences between $m$~multiway Dowker complexes and a variant of the rectangle complex of Brun and Salbu from their recent short proof of Dowker's Theorem. Our proof uses Smale's homotopy mapping theorem and factors through a `cellular Dowker lemma' that expresses the main idea in more general form.
To make the geometry more transparent, we work with a class of spaces called `prod-complexes' then transfer the results to simplicial complexes through a `simplexification' process. We conclude with a detailed study of ternary relations, identifying seven functorially defined homotopy types and twelve natural transformations between them.
\end{abstract}

\maketitle

\setcounter{tocdepth}{2}
\tableofcontents

\section{Introduction}
\label{sec:introduction}

\subsection{Background}
\label{subsec:motivation}

This paper originated as a breakout-group project at the workshop {\it Applied homological algebra beyond persistence diagrams} (June 19--23, 2023) hosted by the American Institute of Mathematics, then located in San Jose, California. The workshop was organized by Chad Giusti, Gregory Henselman-Petrusek, and Lori Ziegelmeier.

We recall that Dowker's Theorem~\cite{Dowker,Bjorner:1995,Chowdhury_Memoli} asserts a functorial homotopy equivalence between two simplicial complexes constructed from a binary relation $R \subseteq I \times J$ between two sets~$I,J$. Our task was to generalize the theorem to higher-order relations $R \subseteq I_1 \times \dots \times I_m$.

In this paper we provide such a generalization, the \emph{multiway Dowker theorem} (\ref{thm:multiway-main-theorem}, \ref{thm:multiway-simplicial-final}).
Adapting Brun and Salbu's approach~\cite{Brun_Salbu_2022} to the original theorem, the multiway version asserts functorial homotopy equivalences between a `relational product' space defined symmetrically with respect to the sets $I_1, \dots, I_m$ and each of $m$ quotient spaces obtained from it by collapsing the sets~$I_k$.
See the top two rows of Figure~\ref{fig:rco-iterated-all}.

\begin{figure}
\centering
\includegraphics[scale=1]{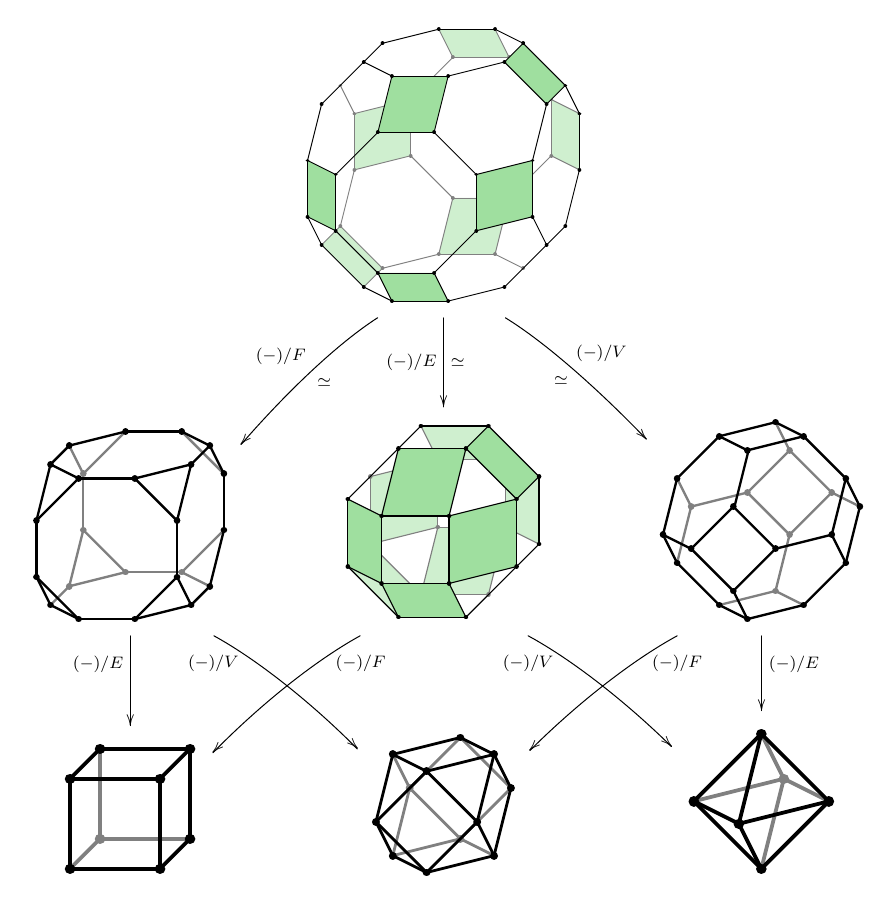}
\caption{A multiway relation $R \subseteq V \times E \times F$ is represented geometrically by a `relational product' (top) which is collapsed by successively quotienting out the factors $V,E,F$ (middle and lower rows).
The multiway Dowker theorem asserts that the top-level collapses are homotopy equivalences. Subsequent collapses need not be.
See Figures \ref{fig:rco}, \ref{fig:rco-iterated}, \ref{fig:rco-projected}, \ref{fig:ternary-natural-rco} for more on this example.
}
\label{fig:rco-iterated-all}
\end{figure}

The theorem may be expressed in various geometrical categories, simplicial complexes being the most familiar.
In the present work we advocate for `prod-complexes': subcomplexes of finite products of simplicial complexes. They form a subclass of the class of prod-simplicial complexes described in~\cite{Kozlov_2008}.
We believe that prod-complexes embody a particularly simple geometry for the theorem. Dowker's original result for two-way relations remains unchanged, its simplicial complexes being prod-complexes of order one. 
The new context illuminates the original theorem: the homotopy equivalence between those simplicial complexes factors naturally through a prod-complex of order two.
That said, we appreciate the need to connect with familiar practice (partly with a view to data science applications). A simplicial complex version of the main theorem is deduced by means of a `simplexification' functor that replaces each prod-complex with a simplicial complex that is homotopy equivalent.

For a brisk exposition, we adopt Smale's homotopy Vietoris mapping theorem~\cite{Smale_1957} as the technical mechanism for establishing homotopy equivalence. This approach leans point-set topological rather than algebraic or combinatorial. We formulate a `cellular Dowker lemma' for CW-complexes from which we derive the multiway Dowker theorem as a special case.
Certainly there exist other arguments leading to the same outcome. We plan to discuss some of these in future work.

\subsubsection*{Outline of the paper.} In the remainder of section~\ref{sec:introduction}, we give a brief history of Dowker's theorem and present the motivating example for the multiway theorem. In section~\ref{sec:CW-Smale}, we review CW-complexes and Smale's mapping theorem. In section~\ref{sec:prod}, we develop the basic facts about prod-complexes. In section~\ref{sec:multiway}, we revisit the original Dowker theorem and present the cellular Dowker lemma and the multiway Dowker theorem for prod-complexes. In section~\ref{sec:simplexification}, we study the simplexification functor and derive the multiway Dowker theorem for simplicial complexes. In section~\ref{sec:ternary}, we systematically organize various Dowker constructions associated to a ternary relation, identifying seven functorially defined homotopy types and twelve natural transformations between them.

\subsection{Brief history of Dowker's theorem}
\label{subsec:history}

C.H.~Dowker considered the following situation in his 1952 paper, `Homology groups of relations'~\cite{Dowker}. Let $R \subseteq I \times J$ be a binary relation. Dowker defined two simplicial complexes, on the sets $I,J$ respectively, which we call
\begin{align*}
    \Dowker_I(R)
    &=
    \{
    \text{non-empty finite $\sigma \subseteq I$}
    \mid
    \text{there exists $j \in J$ such that $\sigma \times \{j\} \subseteq R$}
    \}
    \\
    \Dowker_J(R)
    &=
    \{
    \text{non-empty finite $\tau \subseteq J$}
    \mid
    \text{there exists $i \in I$ such that $\{i\} \times \tau \subseteq R$}
    \}
\end{align*}
(for example Figure~\ref{fig:RGB}) and demonstrated that they have isomorphic homology.
This result and its generalizations have come to be known as `Dowker duality'.
Bj{\o}rner's 1995 book chapter~\cite{Bjorner:1995} gives a short proof along the following lines: each Dowker complex is equal to the {nerve} of a good cover of the other, so the result follows (twice, in fact) from the nerve theorem.

\begin{figure}
\centering
\includegraphics[scale=1]{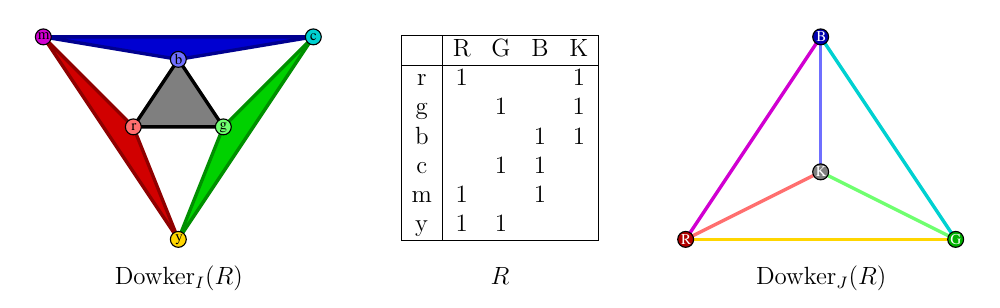}
    \caption{A binary relation and its Dowker complexes.
    }
    \label{fig:RGB}
\end{figure}

The isomorphism in homology arises, as expected, from a homotopy equivalence
\[
\xymatrix{
|\Dowker_I(R)|
    \ar@{<->}[r]^-{\simeq}
&
|\Dowker_J(R)|
}
\]
between the geometric realizations of the Dowker complexes~\cite{Bjorner:1995}. In 2018, Chowdhury and M\'{e}moli \cite{Chowdhury_Memoli} showed that this homotopy equivalence is functorial
in that we have a diagram
\[
\xymatrix{
|\Dowker_I(R)|
    \ar@{<->}[r]^-{\simeq}
    \ar[d]
&
|\Dowker_J(R)|
    \ar[d]
\\
|\Dowker_I(R')|
    \ar@{<->}[r]^-{\simeq}
&
|\Dowker_J(R')|
}
\]
that commutes up to homotopy when $R \subseteq R'$.
This statement is tailored to a key implication in topological data analysis:
Let $r(i,j)$ be any real-valued function on $I \times J$, and let
\[
R^t
= r^{-1} (-\infty,t]
= \{ (i,j) \in I \times J \mid r(i,j) \leq t \}
\]
for $t \in \R$. Then the {nested families} of simplicial complexes
\[
\big( \Dowker_I(R^t) \big)_{t \in \R}
\quad\text{and}\quad
\big( \Dowker_J(R^t) \big)_{t \in \R}
\]
have isomorphic persistence diagrams in each dimension.
Stronger functoriality statements can be found in Remark~\ref{rem:classical-functoriality} and in the remarkable paper of Virk~\cite{Virk_2021}.

The arguments in \cite{Dowker}, and especially \cite{Chowdhury_Memoli}, have a somewhat intricate look. The root cause is that the homotopy equivalence is not realizable by simplicial maps
\[
\xymatrix{
\Dowker_I(R)
    \ar@<3pt>[r]
&
\Dowker_J(R)
    \ar@<3pt>[l]
}
\]
but is constructed using barycentric subdivisions. The subsequent validation becomes surprisingly complicated for such a simple elegant result.
In 2022, Brun and Salbu~\cite{Brun_Salbu_2022} resolved this anomaly by constructing, functorially, a diagram of simplicial maps
\[
\xymatrix{
\Dowker_I(R)
&
\rect(R)
    \ar[r]
    \ar[l]
&
\Dowker_J(R)
}
\]
mediated by a simplicial `rectangle complex' defined on the vertex set $I \times J$.
Each map is a homotopy equivalence on geometric realizations. As is often the case, the \emph{inverse} homotopy equivalences cannot be realized by simplicial maps. Retrospectively it becomes apparent that complications originate from attempts to construct a map directly from one Dowker complex to the other, rather than through a suitable intermediary.

More recently, in 2024, Brun and Grinberg~\cite{Brun_Grinberg_2024} and, independently, one of the authors of the present paper Yoon~\cite{Yoon_2024} constructed an intermediate space (the `biclique complex' or `relational join') with homotopy equivalences induced by simplicial maps \emph{from} the two Dowker complexes to it. Again the inverse homotopy equivalences are not simplicial. Two other proofs are given in~\cite{Yoon_2024}, one based on poset adjunctions and the other involving a `relational product' that plays a central role in the present paper.
See Figure~\ref{fig:rhombicuboctahedron} for a (possibly unhelpful) preview of this last construction.

\begin{figure}
\centering
\includegraphics[scale=1]{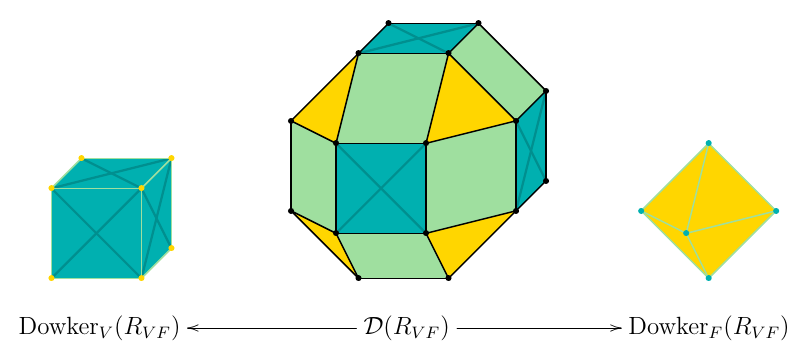}
    \caption{Let $V, F$ be the sets of vertices and faces of a cube. Let $R_{VF}$ be the set of pairs $(v,f)$ with $v$ in~$f$.
    Then $\Dowker_V(R_{VF})$ is a cube with its squares `thickened' into tetrahedra, $\Dowker_F(R_{VF})$ is an octahedron, and the relational product $\dowker(R_{VF})$ is a rhombicuboctahedron with the six `cubic' squares thickened.
    All are homotopy equivalent to $S^2$.
    }
    \label{fig:rhombicuboctahedron}
\end{figure}

\subsection{Motivating example}
\label{subsec:hexagon-example}

We consider the possibilities for a multiway Dowker theorem. A ternary relation between three sets $I, J, K$ is a subset $R \subseteq I \times J \times K$. What geometrical objects does it encode and what are the relationships between those objects?

A first guess would be to define three simplicial complexes on the vertex sets $I, J, K$:
\begin{align*}
\Dowker_I(R)
&=
\{
\sigma \subseteq I
\mid
\text{there exist $j \in J$, $k \in K$ such that $\sigma \times \{j\} \times \{k\} \subseteq R$}
\}
\\
\Dowker_J(R)
&=
\{
\tau \subseteq J
\mid
\text{there exist $i \in I$, $k \in K$ such that $\{i\} \times \tau \times \{j\} \subseteq R$}
\}
\\
\Dowker_K(R)
&=
\{
\upsilon \subseteq K
\mid
\text{there exist $i \in I$, $j \in J$ such that $\{i\} \times \{j\} \times \upsilon \subseteq R$}
\}
\end{align*}
This is quickly seen to be a failure. Let $R_*$ be any relation on sets $I,J$, let $K = \{*\}$ be a singleton set, and define $R = \{ (i,j,*) \mid (i,j) \in R_* \}$. Then
\[
\Dowker_I(R) = \Dowker_I(R_*)
\quad \text{and} \quad
\Dowker_J(R) = \Dowker_J(R_*)
\]
whereas $\Dowker_K(R)$ is a single point. The first two spaces, homotopy equivalent by Dowker's theorem, can have any simplicial homotopy type; but there is no information in the third space. Collapsing two of the factors $I,J,K$ is enough to destroy homotopy information. See also section~\ref{subsec:iterated}.

The following example points the way forward. Let $I = J = K$ be doubleton sets. The 3-dimensional array
\[
\left[
\begin{array}{cc|cc}
    1 & 0 & 1 & 1 \\
    1 & 1 & 0 & 1
\end{array}
\right]
\]
represents the indicator function of a relation $R \subseteq I \times J \times K$ containing six out of the possible eight elements.
Instead of constructing three spaces on the vertex sets $I, J, K$ we construct a central space on the vertex set $I \times J \times K$ along with projection maps onto three spaces on the respective vertex sets $J \times K$, $I \times K$, $I \times J$. The central object has an edge for each pair of adjacent 1s in the array.
See Figure~\ref{fig:hexagon}.

\begin{figure}
    \centering
    \includegraphics[width=0.6\linewidth]{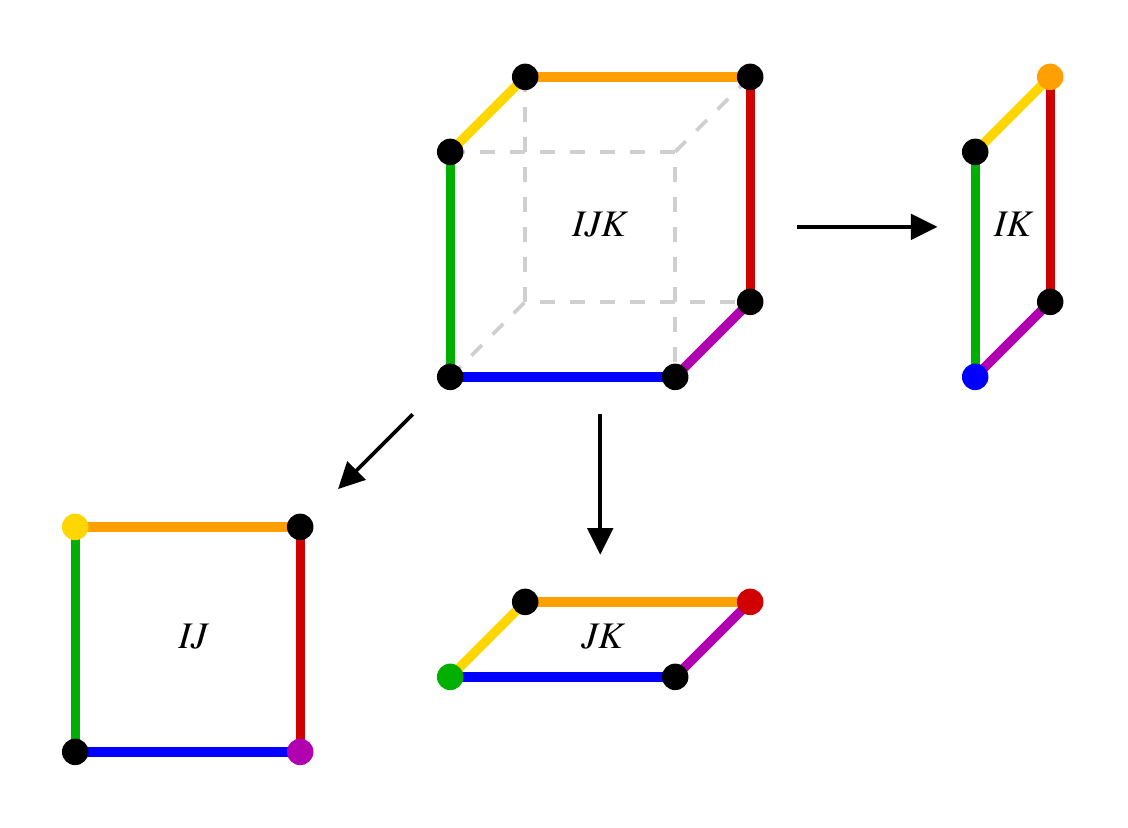}
    \caption{A cell complex $\cuboid_{IJK}$ built from a three-dimensional array with projections onto three cell complexes $\Dowker_{JK}$, $\Dowker_{IK}$, $\Dowker_{IJ}$. Each projection is a homotopy equivalence from a hexagon to a square.
    }
    \label{fig:hexagon}
\end{figure}

A similar structure applies to $k$-way relations: we have a product-like central object along with $k$~projection maps, each a homotopy equivalence collapsing one of the factors.

It remains to settle the details of the construction and confirm that the projection maps are the homotopy equivalences we suspect them to be. Since we are constructing complexes whose vertex sets are products of sets, we are guided towards the use of prod-complexes. The central object is defined in section~\ref{subsec:multiway-prod} and the three projections are obtained through a quotienting process described earlier in section~\ref{subsec:dowkerian-quotients}.
For the simplicial complex version, the reader may skip to section~\ref{subsec:multiway-simplicial} to get straight to those definitions.

\section{CW-complexes and Smale's theorem}
\label{sec:CW-Smale}

In this section we review the basic theory of CW-complexes \cite{Whitehead_1949,Hatcher_2002} and present a theorem of Smale~\cite{Smale_1957} that yields homotopy equivalences of finite CW-complexes. To complete the toolbox, we give a protocol for extending homotopy equivalence results from finite CW-complexes to infinite CW-complexes.

Cell complexes generally come with a partial order on their cells. Subcomplexes correspond to downwardly closed subsets, or `down-sets'. The following notation is convenient: if $A$ is a poset and $B$ is a subset, we write $B \leq A$ to indicate that $B$ is a down-set.

\subsection{CW-complexes}
\label{subsec:CW}

J.H.C.~Whitehead~\cite{Whitehead_1949} introduced CW-complexes in 1949 as a class of cell complex sufficiently versatile for the demands of homotopy theory. We recommend Hatcher~\cite[Appendix~1]{Hatcher_2002} or Brown~\cite[Section~4.7]{Brown_2006} for an exposition of the basic theory. In this section we review the notation and properties that we need later on.

A CW-complex $X$ is an identification space of a family of closed disks $(\overline{D}_\alpha)$ of various dimensions, glued together according to certain rules \cite{Whitehead_1949,Hatcher_2002,Brown_2006}. The rules imply that the set~$X$ is partitioned into `open cells' $e_\alpha$ that are the homeomorphic image of the open disks $D_\alpha = \operatorname{int}(\overline{D}_\alpha)$. Thus we have a diagram
\[
\xymatrix{
\displaystyle
\coprod_{\alpha \in A}
D_\alpha
    \ar[r]
&
\displaystyle
\coprod_{\alpha \in A}
\overline{D}_\alpha
    \ar@{>>}[r]
&
X
}
\]
where the surjective second arrow is a quotient map (a subset of~$X$ is closed iff its pullback to each $\overline{D}_\alpha$ is closed) and the composite map is a continuous bijection that is a homeomorphic embedding on each summand (but not itself a homeomorphism except in trivial cases).
The quotient map restricts, for each~$\alpha$, to a map carrying the closed disk $\overline{D}_\alpha$ onto the closure $\overline{e}_\alpha$ of~$e_\alpha$. This map need not be a bijection; if it is then it is a homeomorphism and we call $\overline{e}_\alpha$ a `closed cell'. In general we call the compact set $\overline{e}_\alpha$ a `cell closure'.
A CW-complex is called \emph{regular} if every $\overline{e}_\alpha$ is a closed cell.

The \emph{carrier} of a point $x \in X$ is the unique $e_\alpha$ containing~$x$.
We caution the reader that $e_\alpha$, although called an `open cell', is not an open subset of~$X$ unless $\alpha$ is a maximal element in the partial order we now describe.

The partial order on the indexing set~$A$ is defined to be the smallest transitive relation that includes $\alpha \geq \alpha'$ whenever $\overline{e}_\alpha \cap e_{\alpha'} \ne \emptyset$. It is a consequence of the construction rules that $\alpha > \alpha'$ implies $\dim(D_\alpha) > \dim(D_{\alpha'})$ ensuring that this relation, evidently reflexive and transitive, is antisymmetric.

A subcomplex of $X$ is a subspace
\[
X_P
=
\bigcup_{\alpha \in P} \overline{e}_\alpha
=
\bigcup_{\alpha \in P} {e}_\alpha
\]
where $P \leq A$ is a down-set. This space has the same topology whether regarded as a subspace of~$X$ or as a quotient space of the disjoint union of closed disks $\coprod_{\alpha \in P} \overline{D}_\alpha$.

\begin{remark}
\label{rem:CW-properties}
We use the following properties of CW-complexes; see \cite[Appendix 1]{Hatcher_2002} or \cite[Section~4.7]{Brown_2006} or the original source~\cite{Whitehead_1949} for proofs and additional context.

\begin{enumerate}
    \item 
    CW-complexes are locally contractible Hausdorff spaces.
    \item
    Finite CW-complexes are compact and metrizable \cite[Corollary A10]{Hatcher_2002}.
    \item
    Every compact subset of a CW-complex is contained in a finite subcomplex.
    Specializing to singleton subsets, every CW-complex is the union of its finite subcomplexes.
    Specializing to cell closures, every finite set of cells is contained in a finite subcomplex.
    \item
    A subset of a CW-complex is closed 
    iff its intersection with every cell closure is closed
    iff its intersection with every finite subcomplex is closed
    iff its intersection with every compact set is closed.
    \item
    Whitehead's theorem: a weak homotopy equivalence of CW-complexes is a homotopy equivalence \cite[Theorem~1]{Whitehead_1949}.
\end{enumerate}
    
\end{remark}

We consider products. Given two CW-complexes presented as quotient spaces
\[
\xymatrix{
\displaystyle
\coprod_{\alpha \in A}
\overline{D}_\alpha
    \ar@{>>}[r]
&
X
}
\quad
\text{and}
\quad
\xymatrix{
\displaystyle
\coprod_{\beta \in B}
\overline{D}_\beta
    \ar@{>>}[r]
&
Y
}
\]
we get a continuous surjection
\[
\xymatrix{
\displaystyle
\coprod_{(\alpha,\beta) \in A \times B}
\overline{D}_\alpha \times \overline{D}_\beta
    \ar@{>>}[r]
&
X \times Y
}
\]
from the disjoint union of the closed disks $\overline{D}_{(\alpha, \beta)} = \overline{D}_\alpha \times \overline{D}_\beta$ onto $X \times Y$ with the product topology.
If we give $X \times Y$ the \emph{quotient topology} (induced from the topology on the disjoint union of closed disks, and equal to or finer than the product topology) it turns out that this yields a CW-complex structure on $X \times Y$.

\begin{remark}
    \label{rem:CG-product}
    The subtlety is that the quotient topology need not be the same as the product topology. In certain favourable circumstances (such as when at least one of $X,Y$ is finite or locally finite~\cite{Whitehead_1949}) the two topologies agree, but in general the quotient topology can be strictly finer~\cite{Dowker_1952b}.
    It is natural to hope that the quotient topology can be constructed independently of the CW-complex representations of $X$ and $Y$.
    There is an affirmative answer from the theory of \emph{compactly generated spaces} \cite[Appendix~1]{Hatcher_2002}: the quotient topology agrees with the `compactly generated product topology' in all cases.
\end{remark}

The same discussion applies to all finite products.
We adopt the following convention:

\begin{convention}
    We work in the category of compactly generated Hausdorff (CGH) spaces. This includes all locally compact Hausdorff spaces and all CW-complexes. The expression $X_1 \times \dots \times X_m$ will denote the product of topological spaces $X_1, \dots, X_m$ with the compactly generated topology rather than the usual product topology. When the spaces are locally compact the two topologies agree. In general the set of maps from a CGH space to a finite product remains equal to the product of the sets of maps to the individual factors.
\end{convention}

For finite complexes the convention changes nothing. Its sole purpose is to allow us to handle infinite complexes---in particular, finite products of infinite complexes---correctly.

\subsection{Smale's mapping theorem}
\label{subsec:smale}

There is a long history of theorems which assert that a map of spaces $f : X \to Y$ whose point fibers $f^{-1}(y)$ are highly connected is itself highly connected (homologically or homotopically). The archetype is the Begle--Vietoris mapping theorem whose conclusion relates the reduced Vietoris homology of the two spaces
\cite{Vietoris_1927,Begle_1950}.

We want a homotopy theorem. In 1957, Smale \cite{Smale_1957} proved the following:

\begin{theorem}[Smale's mapping theorem]
    \label{thm:smale-v1}
    Let $f : X \to Y$ be proper and onto where $X$ and $Y$ are 0-connected, locally compact, separable metric spaces, $X$ is $LC^n$, and for each $y \in Y$, $f^{-1}(y)$ is $LC^{n-1}$ and $(n-1)$-connected. Then
    \\
    \null\quad
    (A) $Y$ is $LC^n$ and
    \\
    \null\quad
    (B) the induced homomorphism $f_{\sharp} : \pi_r(X) \to \pi_r(Y)$ is an isomorphism onto for all $0 \leq r \leq n-1$ and onto for $r=n$.
    \qed
\end{theorem}

The $LC^n$ condition on a space may be unfamiliar;
it means that every neighborhood of a point contains a sub-neighborhood into which every map of a $k$-sphere, for $k = 0, \dots, n$, is null-homotopic in the larger original neighborhood.
It is intermediate in strength between two perhaps more familiar conditions:
\[
\text{locally $n$-connected}
\;\; \Rightarrow \;\;
LC^n
\;\; \Rightarrow \;\;
\text{semi-locally $n$-connected}
\]
In particular if a space is locally contractible, meaning that every neighborhood of a point contains a contractible sub-neighbourhood, then it is $LC^n$ for all~$n$.

The proof of (B) assuming (A) is direct and contained entirely within Smale's paper.
The proof of (A) follows a circuitous path that invokes much external technology (including a body of work that Smale calls the `Begle--Vietoris theory'). That part of the argument can be circumvented by presupposing $Y$ to be locally contractible, as it is in our applications.

Here is a convenient formulation incorporating these remarks:

\begin{theorem}[Smale's mapping theorem, version 2]
    \label{thm:smale-v2}
    Let $f : X \to Y$ be a continuous map between locally contractible compact metrizable spaces (such as finite CW-complexes). Suppose each fiber $f^{-1}(y)$ is contractible (in particular non-empty) and locally contractible. Then $f$ is a weak homotopy equivalence. In particular if $X,Y$ are finite CW-complexes then $f$ is a homotopy equivalence.
\end{theorem}

\begin{proof}
    Suppose first that $X,Y$ are path-connected. Theorem~\ref{thm:smale-v1} immediately implies that $f$ induces isomorphisms of all homotopy groups, making it a weak homotopy equivalence.
    
    In general, since $X,Y$ are compact and locally path-connected they have finitely many connected components; each of these is compact and path-connected. 
    Let $Y'$ be a component of~$Y$. Then $f^{-1}(Y')$ is a union of components of $X$. The images of these components under $f$ form a disjoint (since the $f^{-1}(y)$ are connected) cover (since the $f^{-1}(y)$ are non-empty) of $Y'$ by finitely many compact connected sets. But $Y'$ is connected so there must be exactly one set in this cover.
    Thus $f$ induces a bijection between the components of $X$ and the components of $Y$. Applying Theorem~\ref{thm:smale-v1} component-wise, we find that $f$ is a weak homotopy equivalence.

    The last assertion follows from Whitehead's theorem that a weak homotopy equivalence between CW-complexes is a homotopy equivalence~\cite[Theorem~1]{Whitehead_1949}.
\end{proof}

\subsection{Infinite complexes}
\label{subsec:infinite}

We include a technical result to smooth the narrative later on. This is a finite-to-infinite principle which allows us to deduce a homotopy equivalence between infinite CW-complexes from homotopy equivalances between finite subcomplexes.
Readers primarily interested in finite Dowker complexes may skip this section.

\begin{proposition}
    \label{prop:finite-to-infinite}
    Let $\phi : X \to Y$ be a continuous map between CW-complexes. Let $(X_\beta), (Y_\beta)$ be families of finite subcomplexes of $X,Y$ with the following properties:
    \begin{itemize}
        \item
        For all $\beta$, the map $\phi$ restricts to a homotopy equivalence $\phi_\beta : X_\beta \to Y_\beta$.
        \item
        For all $F,G$ finite subcomplexes of $X,Y$ there exists $\beta$ such that $F \subseteq X_\beta$ and $G \subseteq Y_\beta$.        
    \end{itemize}
    Then $\phi$ is a homotopy equivalence.
\end{proposition}

Since every compact subset of a CW-complex is contained in a finite subcomplex,
the second condition implies that we can cover arbitrary compact subsets $F \subseteq X$, $G \subseteq Y$ by $X_\beta, Y_\beta$ for some~$\beta$.

\begin{proof}
    By the Whitehead theorem it suffices to show that each homomorphism
    \[
    \pi_n(\phi,x) : \pi_n(X,x) \to \pi_n(Y,\phi(x))
    \]
    is surjective and injective. Both parts of the proof use the commutative diagram
    \[
    \xymatrix{
    \pi_n(X,x)
        \ar[rrr]^{\pi_n(\phi,x)}
    &&&
    \pi_n(Y,\phi(x))
    \\
    \pi_n(X_\beta,x)
        \ar[u]
        \ar[rrr]^{\pi_n(\phi_\beta,x)}_{\cong}
    &&&
    \pi_n(Y_\beta,\phi(x))
        \ar[u]
    }
    \]
    for $\beta$ chosen as needed. We write $D$ for the closed $n$-disk.

    For surjectivity, consider an element $[g] \in \pi_n(Y,\phi(x))$ represented by a map $g : D \to Y$. The image $g(D)$ is a compact subset of $Y$ and therefore contained in some $Y_\beta$ where it represents a class $[g]_\beta \in \pi_n(Y_\beta, \phi(x))$.
    Since the lower arrow is surjective, this is the image of a class $[f]_{\beta} \in \pi_n(X_\beta,x)$. Then $[f] \in \pi_n(X,x)$ maps to $[g]$.

    For injectivity, suppose $[f_1], [f_2] \in \pi_n(X,x)$ with $[\phi f_1] = [\phi f_2]$ in $\pi_n(Y,\phi(x))$. Then there is a homotopy $H: D \times [0,1] \to Y$ between $\phi f_1$, $\phi f_2$. Select $\beta$ so that $X_\beta$ contains the compact set $f_1(D) \cup f_2(D)$ and $Y_\beta$ contains the compact set $H(D \times [0,1])$.
    Then $[\phi f_1]_\beta = [\phi f_2]_\beta$ in $\pi_n(Y_\beta, \phi(x))$. Since the lower arrow is injective, this implies that $[f_1]_\beta = [f_2]_\beta$ in $\pi_n(X_\beta,x)$ and hence $[f_1] = [f_2]$ in $\pi_n(X,x)$.
\end{proof}

\section{Prod-complexes}
\label{sec:prod}

We use the term `{prod-complex}' to refer to a subcomplex of a finite product of simplicial complexes. In this section we develop the basic properties of these cell complexes.

\subsection{Simplicial complexes}
\label{subsec:simp}

We begin by reviewing simplicial complexes. The primary object is the {abstract simplicial complex} (we typically omit the word `abstract'). We associate to it a {geometric realization}. This is a regular CW-complex, and may sometimes be called a geometric simplicial complex (we typically do not omit the word `geometric'). A map of simplicial complexes induces a continuous map of geometric realizations.

\subsubsection*{Simplicial complexes.}
For any set~$I$, we define the \emph{complete simplex}\footnote{%
When $I$ is infinite it may be preferable to call it the `complete simplicial complex' on~$I$.
}
\emph{on~$I$} to be the set
\[
\Delta I
=
\{
\sigma \subseteq I \mid \text{$\sigma$ is non-empty and finite}
\}
\]
partially ordered by inclusion.
A \emph{simplicial complex} on~$I$ is defined to be a down-set $S \leq \Delta I$. The elements of~$S$ are its \textit{simplices}. A simplex $\sigma$ with $k+1$ elements is called a $k$-simplex, with dimension $\dim(\sigma)=k$. If~$\tau \subseteq \sigma$ we call~$\tau$ a \textit{face} of~$\sigma$. The down-set condition amounts to the requirement that a simplicial complex contains every face of every simplex in it.

The vertices of a simplex are its elements. The vertices of a simplicial complex $S$ are the vertices of all the simplices of~$S$. We may write the set of vertices as $\vtc(S) = \bigcup S = \bigcup_{\sigma \in S} \sigma$. This set stands in bijective correspondence with the set of singleton elements, the 0-simplices of~$S$.
It is not uncommon to require $\vtc(S) = I$, but we do not impose this condition. In persistent topology we typically consider filtered families of simplicial complexes where earlier members of a family may have fewer vertices than later ones; we may wish to use the same underlying set~$I$ for all complexes in the family.

\begin{example}
    \label{ex:standard-n-simplex}
    The standard $n$-simplex is $\Delta^n = \Delta\{0, 1, \dots, n\}$. We can regard the family $(\Delta^n)$ as simplicial complexes on the natural numbers $\N$, with $\vtc(\Delta^n) = \{0, 1, \dots, n\}$.
\end{example}

\subsubsection*{Geometric realization.}
To a simplicial complex $S$ we associate a CW-complex $|S|$ constructed in the following way.
Define vector spaces\footnote{%
The index being raised or lowered respectively indicates a contravariant or covariant dependence on~$I$.
}
\begin{align*}
\R^{I}
&=
\{
\text{indexed families $(x_i) = (x_i)_{i \in I}$ of real numbers}
\},
\\
\R_I
&=
\{
\text{$(x_i) \in \R^I$ such that at most finitely many of the $x_i$ are non-zero}
\}.
\end{align*}
The second of these is the `free vector space' on~$I$: the standard basis vectors associated to the elements of~$I$ constitute a basis for $\R_I$.
If $I$ is finite the two vector spaces are equal, and we give $\R_I = \R^I$ the Euclidean topology. In general we assign the direct-limit topology to $\R_I$ so that a set is closed iff its intersection with $\R_F$ is closed
for each finite subset $F \subseteq I$.
We note that this is a compactly generated topology, and that
all linear maps are continuous.\footnote{%
Since every finite-dimensional subspace of $\R_I$ is contained in some $\R_F$, a set is closed iff its intersection with each finite-dimensional subspace is closed. It readily follows that all linear maps are continuous.}

Let $\basis : I \to \R_I$ denote the map that carries $i \in I$ to the corresponding standard basis vector in~$\R_I$.
Each simplex $\sigma \in S$ has a geometric realization
\[
|\sigma|
=
\conv(\basis(\sigma))
=
\conv \{ \basis(i) \mid i \in \sigma \}
\]
that is the convex hull of the standard basis vectors associated to its vertices. The geometric realization of $S$ is the union of the geometric realizations of its simplices
\[
|S| = \bigcup_{\sigma \in S} |\sigma|
\]
topologized as a subpace of~$\R_I$ with the direct-limit topology.

Comparing this construction with the definitions in \cite{Whitehead_1949,Hatcher_2002} we find that $|S|$ is a regular CW-complex indexed by~$S$, with closed cells $\overline{e}_\sigma = |\sigma|$ of dimension $\dim(\overline{e}_\sigma) = \dim(\sigma)$, and with the CW-complex partial order agreeing with the original partial order on~$S$.

The following concept is useful for certain calculations.
    Let $x = (x_i)$ be a point in $|S|$. Since $x$ lies in the convex hull of standard basis vectors, we have $0 \leq x_i \leq 1$ for all indices and $\sum x_i = 1$.
    The finite set
    \[
    \sigma = \{ i \in I \mid x_i > 0 \}
    \]
    is the unique simplex of~$S$ such that $x$ lies in the open cell $e_{\sigma}$. Thus $e_\sigma$ is the carrier of~$x$ in the sense that we defined for CW-complexes. We find it convenient to refer to $\sigma$ itself as the carrier of~$x$.

\subsubsection*{Simplicial maps.}
A map between simplicial complexes~$S, T$ on sets $I,J$ is a function $f : I \rightarrow J$ that sends simplices to simplices: if~$\sigma \in S$, then~$f(\sigma)\in T$.  
When this condition is satisfied, we write $f : S \to T$ in a slight overloading of notation.

A simplicial map induces a map between geometric realizations.
Indeed, any function $f : I \to J$ defines a linear map $\R_{f} : \R_I \to \R_J$ by
\[
\basis(i) \longmapsto \basis(f(i))
\]
in terms of standard basis vectors, or equivalently
\[
(x_i) \longmapsto (y_j)
\quad
\text{where}
\quad
y_j = \sum (x_i \mid i \in f^{-1}(j))
\]
in coordinates.
Observe that if $\sigma$ is the carrier of $(x_i)$ then $f(\sigma)$ is the carrier of $(y_j)$. It follows that $\R_{f}$ restricts to a map $|f| : |S| \to |T|$ iff $f$ is a simplicial map $S \to T$.

Let $S,T,U$ be simplicial complexes on $I,J,K$, and let $f : S \to T$ and $g : T \to U$. By considering standard basis vectors, we see that
$\R_{g} \circ \R_{f} = \R_{gf}$ and therefore $|g| \circ |f| = |gf|$.
Likewise $\R_{\id_I} = \id_{\R_I}$ and therefore $|\id_S| = \id_{|S|}$.
We conclude that geometric realization is functorial from the category of simplicial complexes and simplicial maps to the category of topological spaces and continuous maps.

\subsection{Prod-complexes}
\label{subsec:prod}

We now introduce prod-complexes\footnote{%
We distinguish the class of prod-complexes from the much larger class of prod-simplicial complexes~\cite{Kozlov_2008}. The key difference is that we maintain a product structure on the vertex set that aligns with the product structures on the cells of the geometric realization.
}
and their geometric realizations.
Whereas a simplicial complex is defined on a set~$I$ of potential vertices, a prod-complex is defined on a list of sets $I_1, \dots, I_m$ and has vertices in the product set $I_1 \times \dots \times I_m$.
When there are $m$ sets we say that the prod-complex `has order~$m$'. 
Simplicial complexes constitute the case $m=1$.

A \emph{prod-complex} on sets $I_1, \dots, I_m$ is a downwardly closed subset $P$ of the product poset
$\Delta I_1 \times \dots \times \Delta I_m$.
A typical element 
$
\sigma = (\sigma_1, \dots, \sigma_m)
$
is called a \emph{prod-simplex} and has dimension
\[
\dim(\sigma)
=
\dim(\sigma_1) + \dots + \dim(\sigma_m).
\]
Its vertex set
$
\sigma_1 \times \dots \times \sigma_m
\subseteq
I_1 \times \dots \times I_m
$
can be identified with (but is not the same as) the set of 0-dimensional faces of~$\sigma$.

The \emph{geometric realization} of a prod-complex on $I_1, \dots, I_m$ is a subcomplex of the product CW-complex $|\Delta I_1| \times \dots \times |\Delta I_m|$.
We associate to each prod-simplex $\sigma = (\sigma_1, \dots, \sigma_m)$ the closed cell
$
|\sigma| = |\sigma_1| \times \dots \times |\sigma_m|
$
and define:
\[
|P| = \bigcup_{\sigma \in P} |\sigma|
\]
It is a regular CW-complex, being a subcomplex of a finite product of regular CW-complexes.

\subsection{Prod-maps}
\label{subsec:prod-maps}

There are various kinds of maps between prod-complexes. We limit our discussion to two types, described in this subsection and the next.

Let $P$ be a prod-complex on $I_1, \dots, I_m$ and $Q$ be a prod-complex on $J_1, \dots, J_m$ (both of order~$m$). We define a \emph{prod-map} $P \to Q$ to be a list of functions
\[
f = (f_1, \dots, f_m)
\qquad
\text{where $f_k : I_k \to J_k$ for all~$k$}
\]
that carries prod-simplices to prod-simplices in the following sense:
\[
(\sigma_1, \dots, \sigma_m) \in P
\quad
\Rightarrow
\quad
(f_1(\sigma_1), \dots, f_m(\sigma_m)) \in Q
\]
It follows from this requirement that the continuous map
\[
|f_1| \times \dots \times |f_m|
\;:\;
|\Delta I_1| \times \dots \times |\Delta I_m|
\;\longrightarrow\;
|\Delta J_1| \times \dots \times |\Delta J_m|
\]
restricts to a map $|f| : |P| \to |Q|$.

Prod-maps are composed on each factor separately. 
Let $P,Q,R$ be prod-complexes of order~$m$, on lists of sets $I_k, J_k, K_k$ respectively.
If
\begin{alignat*}{2}
f &= (f_1, \dots, f_m)
&&\,:\,
\xymatrix{
I_1 \times \dots \times I_m
    \ar[r]
&
J_1 \times \dots \times J_m
}
\\
g &= (g_1, \dots, g_m)
&&\,:\,
\xymatrix{
J_1 \times \dots \times J_m
    \ar[r]
&
K_1 \times \dots \times K_m
}
\end{alignat*}
define prod-maps $f : P \to Q$ and $g : Q \to R$ then
\[
gf = (g_1 f_1, \dots, g_m f_m)
\,:\,
\xymatrix{
I_1 \times \dots \times I_m
    \ar[r]
&
K_1 \times \dots \times K_m
}
\]
defines a prod-map $gf: P \to R$.

Geometrically, it is immediately verified that $|gf| = |g| \circ |f|$ and $|\id_P| = \id_{|P|}$. Indeed we may assume that $P,Q,R$ are the complete prod-complexes on their respective vertex sets, then apply the simplicial complex case to each factor separately.
We conclude that geometric realization is functorial from the category of prod-complexes and prod-maps  (of each fixed order~$m$) to the category of topological spaces and continuous maps.

\subsection{Dowkerian quotients}
\label{subsec:dowkerian-quotients}
We name these maps and quotient spaces after Dowker because of their role in his theorem and its multiway generalization. Dowker complexes, it turns out, are special instances of Dowkerian quotient spaces. Dowker's theorem amounts to the fact that the quotient map is a homotopy equivalence in certain cases.

Let $P$ be a prod-complex on $I_1, \dots, I_m$. The $k$-th Dowkerian quotient of~$P$ is the prod-complex $P/I_k$ (alternatively written $P_{/k}$) on $I_1, \dots, \widehat{I}_k, \dots, I_m$ defined as follows:    
\begin{equation}
\label{eq:dowkerian-quotient}
P/I_k
= P_{/k}
=
    \{
    (\sigma_1, \dots, \widehat{\sigma}_k, \dots, \sigma_m)
    \mid
    (\sigma_1, \dots, \sigma_k, \dots, \sigma_m) \in P
    \}
\end{equation}
The hats \, $\widehat{~}$ \, indicate omission; we delete $I_k$ from the list of sets and delete the $k$-th factor of each prod-simplex of~$P$.
There is a map of posets $\psi_k = \psi_{P,k} : P \to P/I_k$ defined by:
\[
\sigma =
(\sigma_1, \dots, \sigma_k, \dots, \sigma_m)
\longmapsto
(\sigma_1, \dots, \widehat{\sigma}_k, \dots, \sigma_m)
=:
\sigma/I_k
\]
We call this the Dowkerian quotient map. Its geometric realization $|\psi_{P,k}| : |P| \to |P/I_k|$ is the restriction of the projection map
\[
|\Delta I_1| \times \dots \times |\Delta I_m|
\;
\longrightarrow
\;
|\Delta I_1| \times \dots \times
|\widehat{\Delta I_k}|
\times \dots \times |\Delta I_m|
\]
that eliminates the $k$-th factor. This carries each closed cell $|\sigma|$ onto the closed cell $|\sigma/I_k|$.

Dowkerian quotients are natural in the following sense.
Let $f = (f_1, \dots, f_m)$ be a prod-map
and write $f_{/{k}} = (f_1, \dots, \widehat{f}_k, \dots f_m)$.
Then we have commutative diagrams
\begin{equation}
\label{eq:dowkerian-naturality}
\xymatrix{
P
    \ar[rr]^{\psi_{P,k}}
    \ar[d]_{f}
&&
P/I_k
    \ar[d]^{f_{/k}}
    \\
Q
    \ar[rr]_{\psi_{Q,k}}
&&
Q/J_k
}
\qquad
\raisebox{-4ex}{\text{and}}
\qquad
\xymatrix{
|P|
    \ar[rr]^{|\psi_{P,k}|}
    \ar[d]_{|f|}
&&
|P/I_k|
    \ar[d]^{|f_{/k}|}
    \\
|Q|
    \ar[rr]_{|\psi_{Q,k}|}
&&
|Q/J_k|
}\end{equation}
of posets and spaces respectively.
The verification is straightforward. For the poset diagram, an element $(\sigma_1, \dots, \sigma_m) \in P$ is mapped to $(f_1(\sigma_1), \dots, \widehat{f_k(\sigma_k)}, \dots, f_m(\sigma_m)) \in Q/J_k$ by either path. The diagram of spaces is a restriction of this commutative diagram of vector spaces:
\[
\xymatrix@R=3pc{
\R_{I_1} \times \dots \times \R_{I_m}
    \ar[r]
    \ar[d]_{\R_{f_1} \times \dots \times \R_{f_m}}
&
\R_{I_1} \times \dots
\times \widehat{\R_{I_k}} \times
\dots \times \R_{I_m}
    \ar[d]^{\R_{f_1} \times \dots \times \widehat{\R_{f_k}} \times \dots \times \R_{f_m}}
\\
\R_{J_1} \times \dots \times \R_{J_m}
    \ar[r]
&
\R_{J_1} \times \dots
\times \widehat{\R_{J_k}} \times
\dots \times \R_{J_m}
}
\]

The following remark will be useful later.

\begin{remark}
    \label{rem:dowkerian-*}
    Dowkerian quotient maps may be interpreted as prod-maps in the following sense. We can identify a prod-complex $P$ over $I_1, \dots, \widehat{I_k}, \dots, I_m$ with a prod-complex $P^*$ over $I_1, \dots, *_k, \dots, I_m$ where $*_k$ denotes a singleton set in the $k$-th position. Indeed, there is a poset isomorphism
    \[
    \Delta I_i \times \dots \times
    \widehat{\Delta {I_k}}
    \times \dots \times \Delta I_m
    \; \cong \;
    \Delta I_i \times \dots \times
    \Delta (*_k)
    \times \dots \times \Delta I_m
    \]
    because $\Delta(*_k)$ is a singleton poset. Moreover $|*_k|$ is a singleton space, so there is a canonical natural isomorphism $|P| \cong |P^*|$.
    Under this correspondence, the Dowkerian quotient map $\psi_{P,k} : P \to P/I_k$ becomes a prod-map $(\id_{I_1}, \dots, *_k, \dots \id_{I_m}) : P \to (P/I_k)^*$ where $*_k$, in a standard abuse of notation, denotes the unique map $I_k \to *_k$.
\end{remark}

\section{The multiway Dowker theorem}
\label{sec:multiway}

In this section we construct the main result.
In section~\ref{subsec:revisited}, we revisit Dowker's theorem in the light of section~\ref{sec:prod}.
This motivates a `cellular Dowker lemma' for arbitrary cell complexes, in section~\ref{subsec:cellular}. In section~\ref{subsec:multiway-prod}, we state and prove the multiway Dowker theorem for prod-complexes. This is primarily an assertion that certain Dowkerian quotient maps are homotopy equivalences.
In section~\ref{subsec:iterated}, we briefly discuss iterated Dowkerian quotients.

\subsection{Dowker's theorem revisited}
\label{subsec:revisited}

Let $R \subseteq I \times J$ be a binary relation.
The definitions
\begin{align*}
    \Dowker_I(R) &=
    \{ \sigma \in \Delta I
    \mid
    \text{there exists $j \in J$ such that
    $\sigma \times \{j\} \subseteq R$}
    \}
    \\
    \Dowker_J(R) &=
    \{ \tau \in \Delta J
    \mid
    \text{there exists $i \in I$ such that
    $\{i\} \times \tau \subseteq R$}
    \}
    \end{align*}
of the two classic Dowker complexes suggest two additional relations:
\begin{alignat*}{2}
    R_I
    &=
    \{ (\sigma, j) \mid
    \sigma \times \{j\} \subseteq R
    \}
    &&\; \subseteq \; \Delta I \times J
    \\
    R_J
    &=
    \{ (i, \tau) \mid
    \{i\} \times \tau \subseteq R
    \}
    &&
    \; \subseteq \; I \times \Delta J
\end{alignat*}
These are downwardly closed, in the following sense:
\begin{alignat*}{3}
(\sigma, j) \in R_I
&\;\; \text{and} \;\;
\sigma' \subseteq \sigma
&\;\; \text{implies} \;\;
(\sigma', j) \in R_I
\\
(i, \tau) \in R_J
&\;\; \text{and} \;\;
\tau' \subseteq \tau
&\;\; \text{implies} \;\;
(i, \tau') \in R_J
\end{alignat*}
This means that the `slices' of these relations
\begin{align*}
R_{I,j}
&= \{ \sigma \in \Delta I \mid (\sigma, j) \in R_I \}
\\
R_{i,J}
&= \{ \tau \in \Delta J \mid (i, \tau) \in R_J \}
\end{align*}
are simplicial complexes for each $j$ and each $i$. Moreover
\[
\Dowker_I(R) = \bigcup_{j \in J} R_{I,j}
\quad \text{and} \quad
\Dowker_J(R) = \bigcup_{i \in I} R_{i,J}
\]
so we recover the two Dowker complexes from these relation slices.

We can repeat the process that gave us $R_I$ and $R_J$ to obtain one more downwardly closed relation. This one we interpret geometrically:

\begin{definition}
    \label{def:relational-product}
    The \emph{Dowker relational product} \cite{Yoon_2024} is the prod-complex on $I,J$ defined:
    \[
    \dowker(R)
    =
    R_{IJ}
    =
    \{
    (\sigma, \tau) \mid \sigma \times \tau \subseteq R
    \}
    \; \subseteq \;
    \Delta I \times \Delta J
    \]
It is the largest prod-complex on $I, J$ whose vertex set is equal to~$R$.
\end{definition}

\begin{proposition}
    \label{prop:DI=D/I}
    The Dowker complexes of~$R$ are equal to the Dowkerian quotients of the Dowker relational product:
    \[
    \Dowker_I(R) = \dowker(R)/J
    \quad \text{and} \quad
    \Dowker_J(R) = \dowker(R)/I   
    \]
\end{proposition}

\begin{proof}
If $\sigma \in \Dowker_I(R)$ then there exists $j \in J$ with $\sigma \times \{j\} \subseteq R$.
Then $(\sigma,\{j\}) \in \dowker(R)$ so $\sigma \in \dowker(R)/J$.
Conversely if $\sigma \in \dowker(R)/J$ then some $\sigma \times \tau \subseteq R$. Taking $j \in \tau$ we have $\sigma \times \{j\} \subseteq R$ and therefore $\sigma \in \Dowker_I(R)$. The second statement is similar.
\end{proof}

\begin{theorem}[Dowker's theorem revisited]
\label{thm:dowker-revisited}
The geometric Dowkerian quotient maps
\[
\xymatrix{
|\Dowker_I(R)| = |\dowker(R)/J|
&&
|\dowker(R)|
    \ar@{>>}[ll]_-{|\psi_2|}
    \ar@{>>}[rr]^-{|\psi_1|}
&&
|\dowker(R)/I| = |\Dowker_J(R)|
}
\]
are homotopy equivalences.
\end{theorem}

This statement can be thought of as a prod-simplicial `squashing' of Brun and Salbu's formulation~\cite{Brun_Salbu_2022}.
We deduce it in the next section from the cellular Dowker lemma \ref{lem:cellular}.

\subsection{The cellular Dowker lemma}
\label{subsec:cellular}

We mimic some of the steps that led to the statement of Theorem~\ref{thm:dowker-revisited} in a more general setting.

Let $X$ be a CW-complex indexed by a set $A$, let $J$ be a set, and let $R \subseteq A \times J$ be a relation that is downwardly closed with respect to the CW-complex partial order on~$A$ and the trivial partial order on~$J$. Thus
\[
(\alpha,j) \in R
\;\; \text{and} \;\;
\alpha' \leq \alpha
\;\; \text{implies} \;\;
(\alpha',j) \in R
\]
for all $\alpha, \alpha' \in A$ and $j \in J$.

We have two types of slice:
\begin{align*}
    R_\alpha &= \{ j \in J \mid (\alpha, j) \in R \}
    \\
    R_j &= \{ \alpha \in A \mid (\alpha, j) \in R \}
\end{align*}
Since $R$ is downwardly closed, the function $\alpha \mapsto R_\alpha$ is order-reversing and each slice $R_j$ is a down-set that defines a subcomplex $X_{R_j}$ of~$X$.
We write
\[
X_R
=
\bigcup_{j \in J} X_{R_j}
=
\bigcup
\,
(
\overline{e}_\alpha 
\mid
\text{there exists $j \in J$ such that $(\alpha,j) \in R$}
)
\]
for the union of these subcomplexes. This is the subcomplex of~$X$ actively involved in~$R$.

The next step is to define a relation between $A$  and $\Delta J$:
\[
R_J
    =
    \{ (\alpha,\tau) \mid
    \{ \alpha \} \times \tau \subseteq R
    \}
    \leq A \times \Delta J
\]
Being a down-set, it defines a subcomplex of $X \times |\Delta J|$
\[
X \times_R |\Delta J|
=
\bigcup
\,(
\overline{e}_\alpha \times |\tau| 
\mid
(\alpha,\tau) \in R_J
)
\]
which we may call the `cellular Dowker relational product' of~$R$.
See Figure~\ref{fig:saucepan}.

\begin{proposition}
\label{prop:cellular-map}
The projection map $X \times |\Delta J| \to X$ restricts to a surjective map:
\[
\xymatrix{
X \times_R |\Delta J|
    \ar[r]
&
X_R
}
\]
\end{proposition}

\begin{proof}
    The image of the map is the union of the closed cells $\overline{e}_\alpha$ for which there exists a simplex~$\tau$ with $(\alpha,\tau) \in R_J$. This condition is equivalent to the existence of a single element~$j$ with $(\alpha,j) \in R$; this is satisfied precisely by the cells of $X_R$.
\end{proof}

\begin{lemma}[cellular Dowker lemma]
    \label{lem:cellular}
    Let $X$ be a CW-complex index by a set~$A$, let $J$ be a set, and let $R \subseteq A \times J$ be a downwardly closed relation, as above. Then the map
    \[
    \xymatrix{
    X \times_R |\Delta J|
        \ar[r]
    &
    X_R
    }
    \]
    of Proposition~\ref{prop:cellular-map} is a homotopy equivalence.
\end{lemma}

\begin{figure}
\centering
\includegraphics[scale=1]{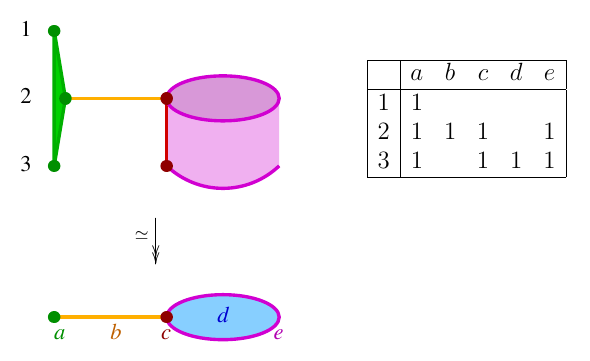}
    \caption{The five cells $a,b,c,d,e$ of the racquet-like space label the columns of a binary relation  whose rows are labeled $1,2,3$.
    The relation is downwardly closed with respect to the cell-complex partial order $a \leq b \geq c \leq e \leq d$ so it defines a cellular Dowker product: a saucepan-like space with a triangular fin at the end of the handle. Its round base is not visible in this picture.
    }
    \label{fig:saucepan}
\end{figure}

\begin{proof}[Proof of Lemma~\ref{lem:cellular} when $X,J$ are finite]
In the finite case, each nonempty $R_\alpha$ is a simplex and we have
\[
X \times_R |\Delta J|
=
\bigcup_{\alpha \in A}
\overline{e}_\alpha \times |R_\alpha|
\]
because $(\alpha,\tau) \in R_J$ iff $\tau \subseteq R_\alpha$.
Let $x \in X_R$.
Then the fiber at~$x$ is
\[
\bigcup_{\alpha \in A_x}
\{x\} \times |R_\alpha|
\]
where $A_x = \{ \alpha \in A \mid x \in \overline{e}_\alpha \}$. Let $e_\beta$ be the carrier of~$x$. Then $x \in \overline{e}_\beta$ so the fiber contains $\{x\} \times |R_\beta|$.
From the definition of the partial order on~$A$, the condition $\alpha \in A_x$ implies $\alpha \geq \beta$, which in turn implies $R_\alpha \subseteq R_\beta$. Thus the fiber is equal to $\{x\} \times |R_\beta|$. 
This is nonempty by Proposition~\ref{prop:cellular-map}, so it is a closed geometric simplex; hence it is contractible and locally contractible.
The Smale mapping theorem~\ref{thm:smale-v2} for finite CW-complexes implies that the map is a homotopy equivalence.
\end{proof}

Readers interested primarily in the finite case may skip the following.

\begin{proof}[Proof of Lemma~\ref{lem:cellular} in full generality]
    We use Proposition~\ref{prop:finite-to-infinite} to extend the result from the finite case.
    Our task is to show that the map
    \[
    \xymatrix{
        X \times_R |\Delta J|
        \ar[r]
        &
        X_R
    }
    \]
    is a homotopy equivalence.
    For each pair $(Y,K)$ where $Y \subseteq X$ is a finite subcomplex and $K \subseteq J$ is a finite subset, we write $S = R \cap (Y \times K)$. We have a diagram
    \[
    \xymatrix{
    X \times_R |\Delta J|
        \ar[rr]
    &&
     X_R
    \\
    Y \times_S | \Delta K|
        \ar[rr]^-{\simeq}
        \ar[u]
    &&
    Y_S
        \ar[u]
    }
    \]
    where the vertical maps are inclusions of finite subcomplexes and the lower map is a homotopy equivalence by the finite case of Lemma~\ref{lem:cellular}.

    Suppose $F \subseteq X \times_R |\Delta J|$ and
    $G \subseteq X_R$
    are finite subcomplexes. The product cells that constitute $F$ involve finitely many cells of $X$, and finitely many simplices of $|\Delta J|$ therefore finitely many elements of $J$. Meanwhile $G$ contains finitely many cells of~$X$; for each cell~$\sigma$ we select $j$ such that $(\sigma,j) \in R$. Choose a finite subcomplex $Y \subseteq X$ and a finite set $K \subseteq J$ that include the finitely many cells of~$X$ and finitely many elements of~$J$ mentioned above. Then $Y \times_S |\Delta K|$ contains~$F$ while
    $Y_S$ contains~$G$.

    The conditions of Proposition~\ref{prop:finite-to-infinite} being satisfied, the map is a homotopy equivalence.    
\end{proof}

The `revisited' Dowker theorem is a special case of the cellular Dowker lemma.

\begin{proof}[Proof of Theorem~\ref{thm:dowker-revisited}]
Let $R, I, J, R_I$ be as in the statement of~\ref{thm:dowker-revisited}.
Apply the cellular Dowker lemma \ref{lem:cellular} to the CW-complex $|\Delta I|$ indexed by $\Delta I$, the set~$J$, and the relation $R_I \subseteq \Delta I \times J$.
The resulting homotopy equivalence is the upper arrow of a commutative diagram
\[
\xymatrix{
|\Delta I|_{R_I}
    \ar@{=}[d]
&&
|\Delta I| \times_{R_I} |\Delta J|
    \ar@{=}[d]
    \ar[ll]_{\simeq}
\\
|\dowker(R)/J|
&&
|\dowker(R)|
    \ar@{>>}[ll]_-{|\psi_2|}
}
\]
where the vertical sides are equalities.
Thus $|\psi_2|$ is a homotopy equivalence. Similarly $|\psi_1|$ is a homotopy equivalence.
\end{proof}

\begin{remark}
\label{rem:classical-functoriality}
Functoriality is immediate.
Let $R \subseteq I_1 \times I_2$ and $S \subseteq J_1 \times J_2$ be relations. Suppose $f = (f_1, f_2)$ where
$f_1 : I_1 \to J_1$,
and
$f_2 : I_2 \to J_2$
are functions with the property that $(i_1, i_2) \in R$ implies $(f_1(i_1), f_2(i_2)) \in S$. Then we have a diagram
\[
\xymatrix{
|\dowker(R)/I_2|
    \ar[d]_{|f_1|}
&&
|\dowker(R)|
    \ar@{>>}[ll]_{|\psi_2|}
    \ar@{>>}[rr]^{|\psi_1|}
    \ar[d]^{|f|}
&&
|\dowker(R)/I_1|
    \ar[d]^{|f_2|}
\\
|\dowker(S)/J_2|
&&
|\dowker(S)|
    \ar@{>>}[ll]_{|\psi_2|}
    \ar@{>>}[rr]^{|\psi_1|}
&&
|\dowker(S)/J_1|
}
\]
where the two squares commute by \eqref{eq:dowkerian-naturality}.
The four corner spaces are geometric simplicial complexes, while the two mediating spaces are geometric prod-complexes of order~2.
\end{remark}

\subsection{Multiway Dowker for prod-complexes}
\label{subsec:multiway-prod}

Here is the main object of our attention.

\begin{definition}[Dowker relational product]
Let $R \subseteq I_1 \times \dots \times I_m$ be a multiway relation. The multiway {Dowker relational product} of~$R$ is the prod-complex on $I_1, \dots, I_m$ defined:
\[
\dowker(R)
=
\{
(\sigma_1, \dots, \sigma_m)
\mid
\sigma_1 \times \dots \times \sigma_m \subseteq R
\}
\leq
\Delta I_1 \times \dots \times \Delta I_m
\]
It is the largest prod-complex on $I_1, \dots, I_m$ whose vertex set is~$R$.
\end{definition}

\begin{theorem}[multiway Dowker theorem]
    \label{thm:multiway-main-theorem}
    Let $R \subseteq I_1 \times \dots \times I_m$ be a multiway relation. Then the geometric Dowkerian quotient map
    $
    \xymatrix{
    |\psi_k| :
    |\dowker(R)| 
        \ar@{>>}[r]
    &
    |\dowker(R)/I_k|
    }
    $
    is a homotopy equivalence for all~$k$.
\end{theorem}

Thus we have a diagram of $m$ homotopy equivalences
    \begin{equation}
    \label{eq:hub-diagram}
    \xymatrix{
    &&& |\dowker(R)|
        \ar@{>>}[llld]_{|\psi_1|}
        \ar@{>>}[ld]^(0.4){|\psi_2|}
        \ar@{>>}[rrrd]^{|\psi_m|}
    \\
    |\dowker(R)/I_1|
    &&
    |\dowker(R)/I_2|
    &&
    \dots
    &&
   |\dowker(R)/I_m|
    }
    \end{equation}
through which the spaces $|\dowker(R)/I_k|$ are canonically homotopy equivalent. See Figures \ref{fig:rco-iterated-all}, \ref{fig:rco}.

\begin{figure}
\centering
\includegraphics[scale=1]{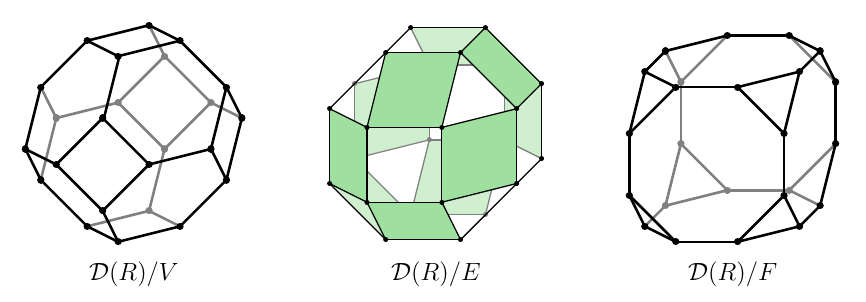}
    \caption{Let $V,E,F$ denote the sets of vertices, edges, and faces of a cube. Let $R$ be the set of triples $(v,e,f)$ such that $v$ is contained in $e$ is contained in~$f$. The three Dowkerian quotients are respectively the 1-skeleton of a truncated octahedron, the union of the twelve `rhombic' squares in a rhombicuboctahedron, and the 1-skeleton of a truncated cube. Each is homotopy equivalent to a 14-times punctured 2-sphere.
    See Figure~\ref{fig:rco-iterated-all} for $\dowker(R)$ itself.
    }
    \label{fig:rco}
\end{figure}

\begin{proof}
We may assume $k=m$.
Apply the cellular Dowker lemma \ref{lem:cellular} with:
\begin{align*}
    X &=
    |\Delta I_1| \times \dots \times |\Delta I_{m-1}|
    \\
    A &=
    \Delta_1 \times \dots \times \Delta_{m-1}
    \\
    J &= I_m
    \\
    \tilde{R} &=
    \{
    ((\sigma_1, \dots, \sigma_{m-1}), j)
    \mid
    \sigma_1 \times \dots \times \sigma_{m-1} \times \{j\}
    \subseteq R
    \}
\end{align*}
The resulting homotopy equivalence is the upper arrow in the commutative diagram
\[
\xymatrix{
\big( 
|\Delta I_1| \times \dots \times |\Delta I_{m-1}| 
\big)
\times_{\tilde{R}} |\Delta I_m|
    \ar[d]_{\cong}
    \ar[rr]^-{\simeq}
&&
\big( |\Delta I_1| \times \dots \times|\Delta I_{m-1}| \big)_{\tilde{R}}
    \ar@{=}[d]
\\
|\dowker(R)|
    \ar[rr]^{|\psi_m|}
&&
|\dowker(R)/I_m|
}
\]
where the right-hand side is an equality and the left-hand arrow is the homeomorphism defined by the following tautological formula:
\[
\xymatrix{
((x_1, \dots, x_{m-1}), x_m)
    \ar@{|->}[r]
&
(x_1, \dots, x_{m-1}, x_m)
}
\]
It follows that $|\psi_m|$ is a homotopy equivalence as claimed.
\end{proof}

\begin{remark}
    \label{rem:multiway-naturality}
    For functoriality, let $S \subseteq J_1 \times \dots \times J_m$ and suppose
    \[
    f = (f_1, \dots, f_m)
    \quad
    \text{where}
    \quad 
    \text{$f_k : I_k \to J_k$ for all $k$}
    \]
    satisfies the condition
    \[
    (i_1, \dots, i_m) \in R
    \; \Rightarrow \;
    (f_1(i_1), \dots, f_m(i_m)) \in S
    \]
    for all $i_1, \dots, i_m$.
    Equation~\eqref{eq:dowkerian-naturality} provides commutative diagrams for all~$k$:
    \[
    \xymatrix{
    |\dowker(R)| 
        \ar@{>>}[rr]^{|\psi_k|}
        \ar[d]_-{|f|}
    &&
    |\dowker(R)/I_k|
        \ar[d]^-{|f_{/k}|}
    \\
    |\dowker(S)| 
        \ar@{>>}[rr]^{|\psi_k|}
    &&
    |\dowker(S)/J_k|    
    }
    \]
    These squares can be combined, like pages in a book along the spine $|\dowker(R)| \stackrel{|f|}{\longrightarrow} |\dowker(S)|$, to confirm that the diagram \eqref{eq:hub-diagram} is functorial in $R$.
\end{remark}

\subsection{Iterated Dowkerian quotients}
\label{subsec:iterated}

The multiway Dowker theorem tells us that we can collapse any one of the factors of the Dowker relational product of a multiway relation without changing the homotopy type. %
What happens if we collapse two or more factors?

It is tempting to simply apply Theorem~\ref{thm:multiway-main-theorem} to each quotient, but we emphasize that this will not work: in most cases $\dowker(R)/I_k$ is not itself a multiway Dowker relational product, so its Dowkerian quotient maps need not be homotopy equivalences.
Figure~\ref{fig:hexagon} contains the simplest example: a square without its interior cannot be a Dowker relational product for any binary relation, and it is not homotopy equivalent to its Dowkerian quotients.

We introduce some notation.
Let $P$ be a prod-complex on $I_1, \dots, I_m$ and $k < \ell$. Defining
\[
P/(I_k, I_{\ell})
=
\{
(\sigma_1, \dots, \widehat{\sigma}_k, \dots, \widehat{\sigma}_{\ell}, \dots, \sigma_m)
\mid
(\sigma_1, \dots, \sigma_m) \in P
\}
\]
it follows that
\[
P/(I_k, I_\ell)
=
(P/I_k)/I_{\ell}
=
(P/I_{\ell})/I_{k}
\]
by iterating \eqref{eq:dowkerian-quotient}.
Higher-order iterated quotients are handled similarly.
See Figure~\ref{fig:rco-iterated} for a trio of examples.

\begin{figure}
\centering
\includegraphics[scale=1]{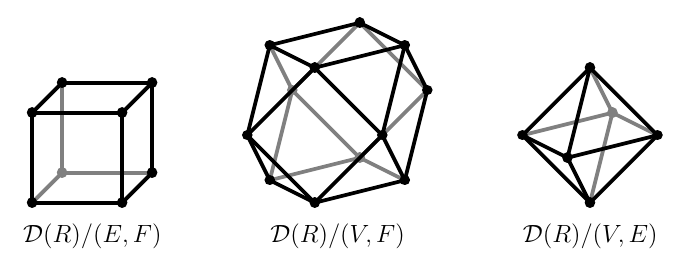}
    \caption{Continuing from Figure~\ref{fig:rco}, the three iterated Dowkerian quotients are respectively the 1-skeleton of a cube, the 1-skeleton of a cuboctahedron, and the 1-skeleton of an octahedron.}
    \label{fig:rco-iterated}
\end{figure}

`Complementary' iterated quotients can be made independent of one another:

\begin{example}
    \label{ex:two-factor-collapse}
    Let $R_1 \subseteq I_1 \times J_1$ and $R_2 \subseteq I_2 \times J_2$ be binary relations. Combining them to obtain a 4-way relation
    \[
    R_1 \times R_2
    \subseteq
    I_1 \times J_1 \times I_2 \times J_2
    \]
    we get
    \[
    \dowker(R_1 \times R_2)/(I_2, J_2)
     =
    \dowker(R_1)
    \quad \text{and} \quad
    \dowker(R_1 \times R_2)/(I_1, J_1)    
    =
    \dowker(R_2).
    \qedhere\]
    Since we can assign arbitrary simplicial complex homotopy types to these two spaces,\footnote{%
    Let $S$ be an arbitrary simplicial complex on~$I$, let $J = S$, and let
    $
    R = \{ 
    (i,\sigma) \in I \times J
    \mid
    i \in \sigma
    \}
    $.
    Then $\dowker(R)/J = S$ and hence $|\dowker(R)| \simeq |\dowker(R)/J| = |S|$.
    }
    there can be no general relation between these iterated quotients.
\end{example}

We have more to say about iterated quotients of \emph{ternary} relations in section~\ref{sec:ternary}.

\section{Simplexification}
\label{sec:simplexification}

We are ready to transfer our results to simplicial complexes.
In section~\ref{subsec:simplexification}, we describe a `simplexification' operation that converts a prod-complex to a simplicial complex, along with a `squashing map' that gives a homotopy equivalence from the simplexification to the prod-complex.
In section~\ref{subsec:simp-squash-fun-natural}, we discuss the functoriality of these constructions and their relationship with Dowkerian quotients. This sets up a clean treatment of the simplicial multiway Dowker theorem in section~\ref{subsec:multiway-simplicial}.

Before we begin, we comment on what simplexification \emph{isn't}.
One way to convert a prod-complex into a simplicial complex is to construct its barycentric subdivision;
its simplices are the finite nonempty chains in the prod-complex viewed as a poset. One advantage is that it is homeomorphic to the original prod-complex when geometrically realized. The main disadvantage is that it has many more low-dimensional cells, making it computationally more expensive. This is true even in the order~1 case where the prod-complex is a simplicial complex, which ideally should not need conversion.
In contrast, simplexification retains the same vertex set and leaves simplicial complexes unchanged.

\subsection{Simplexification and squashing}
\label{subsec:simplexification}

Let $P$ be a prod-complex on $I_1, \dots, I_m$. To each prod-simplex $\sigma = (\sigma_1, \dots, \sigma_m)$ we assign a simplex\footnote{%
We can think of $\simp(\sigma)$ as a geometrically `thickened' version of $\sigma$ with the same vertex set $\sigma_1 \times \dots \times \sigma_m$.
}
\[
\simp((\sigma_1, \dots, \sigma_m))
= 
\sigma_1 \times \dots \times \sigma_m
\; \in \; \Delta(I_1 \times \dots \times I_m).
\]

\begin{definition}[simplexification]  
\label{def:simplexification}
The simplexification of~$P$ is the simplicial complex on $I_1 \times \dots \times I_m$ defined as follows:
\[
\simp(P)
=
\bigcup_{\sigma \in P} \Delta(\simp(\sigma))
\]
Its simplices are characterized by the following property:
\[
\rho \in \simp(P)
\; \Leftrightarrow \;
\text{there exists $(\sigma_1, \dots, \sigma_m) \in P$ such that $\rho \subseteq \sigma_1 \times \dots \times \sigma_m$}
\]
\end{definition}

Note that $\simp(P) = P$ when $m=1$.
We must go to $m=2$ for the archetypal example:

\begin{example}[simplexification of a square]
    Recall that $\Delta^n = \Delta{\{0,\dots,n\}}$ denotes the standard n-simplex.
    Then    
    \[
    Q = \Delta^{1} \times \Delta^{1}
    \]
    is a prod-complex on $\{0,1\}, \{0,1\}$ with geometric realization 
    \[
    |Q| = |\Delta^{1}| \times |\Delta^{1}|
    \]
    homeomorphic to a square;
    whereas
    \[
    \simp(Q) = \Delta(\{0,1\} \times \{0,1\})
    \]
    is a complete simplex on four vertices, with geometric realization
    \[
    |\simp(Q)| \cong |\Delta^{3}|
    \]
    homeomorphic to a tetrahedron.
\end{example}

A tetrahedron viewed from a particular angle has the silhouette of a square.
A photograph of such a view implicitly `squashes' the tetrahedron onto the square $\boxtimes \to \square$.
Guided by this, we give a general definition of squashing map
\[
\xi_P :
\xymatrix{|\simp(P)| \ar[r] & |P|}
\]
in the following way.
Each projection $\pi_k : I_1 \times \dots \times I_m \to I_k$ induces a linear map
\[
\xymatrix{
\R_{I_1 \times \dots \times I_m}
    \ar[rrr]^-{\xi_k = \R_{\pi_k}}
&&&
\R_{I_k}
}
\]
We combine these into a single map $\xi = (\xi_1, \dots, \xi_m)$
and define $\xi_P$ to be the retriction of~$\xi$ indicated in the following diagram (the vertical arrows being the inclusion maps):
\[
\xymatrix{
\R_{I_1 \times \dots \times I_m}
    \ar[rrr]^-{\xi = (\xi_1, \dots, \xi_m)}
&&&
\R_{I_1} \times \dots \times \R_{I_m}
\\
|\simp(P)|
    \ar[u]
    \ar@{>>}[rrr]^-{\xi_P}
&&&
|P|
    \ar[u]
}
\]
Later (but not right now) we may write $\xi$ instead of $\xi_P$ to avoid cumbersome notation.

\begin{proposition}
    \label{prop:squash-well-defined}
    The squashing map $\xi_P$ is well-defined and surjective.
\end{proposition}

\begin{proposition}
    \label{prop:squash-sigma}
    We have $\xi(|\simp(\sigma|) = |\sigma|$ for all $\sigma \in \Delta I_1 \times \dots \times \Delta I_m$.
\end{proposition}

\begin{proof}
    From the formula
    \[
    \xymatrix{
    \basis((i_1, \dots, i_m))
        \ar@{|->}[r]^-{\xi}
    &
    (\basis(i_1), \dots, \basis(i_m))
    }
    \]
    it follows that $\xi$ carries the vertices of $|\simp(\sigma)|$ bijectively onto the vertices of $|\sigma|$. Convex hulls are preserved by linear maps, so $\xi$ carries $|\simp(\sigma)|$ onto $|\sigma|$.
\end{proof}

\begin{proof}[Proof of Proposition~\ref{prop:squash-well-defined}]
    Take the union of Proposition~\ref{prop:squash-sigma} over $\sigma \in P$.
\end{proof}

As a step towards understanding the fibers of $\xi_P$, we identify the preimage of each~$|\sigma|$.

\begin{proposition}
    \label{prop:squash-preimage}
    For each $\sigma \in P$ we have $\xi_P^{-1} |\sigma| = |\simp(\sigma)|$.
\end{proposition}

\begin{lemma}
    \label{lem:squash-carriers}
    Suppose $x \in |\Delta(I_1 \times \dots \times I_m)|$ has carrier $\rho$. Then $\xi(x) \in |\Delta I_1| \times \dots \times |\Delta I_m|$ has carrier $(\pi_1(\rho), \dots, \pi_m(\rho))$.
\end{lemma}

\begin{proof}
Equivalently, $\xi_k(x) \in |\Delta I_k|$ has carrier $\pi_k(\rho)$ for all~$k$. Let $i = (i_1, \dots, i_m)$ denote a typical element of $I_1 \times \dots \times I_n$.
Writing $x$ as a convex combination
\[
x =
\sum_{i \in \rho} \lambda_i \basis((i_1, \dots, i_m)) 
\]
with positive coefficients, we have
\[
\xi_k(x)
=
\sum_{i \in \rho} \lambda_i \basis(i_k)
=
\sum_{j \in I_k}
\mu_j \basis(j)
\quad
\text{where} \;
\mu_j = \sum\big( \lambda_i \mid i \in \rho,\, i_k = j \big)
\]
so the positive coefficients occur precisely at the vertices of $\pi_k(\rho)$.
\end{proof}

\begin{proof}[Proof of Proposition~\ref{prop:squash-preimage}]
By Proposition~\ref{prop:squash-sigma} it suffices to show that $\xi_P^{-1}|\sigma| \subseteq |\simp(\sigma)|$.
Suppose $x \in \xi_P^{-1}|\sigma|$. Let $\rho$ denote its carrier. By the lemma, $(\pi_1(\rho), \dots, \pi_m(\rho))$ is the carrier of $\xi(x)$. Since $\xi(x) \in |\sigma| = |\sigma_1| \times \dots \times |\sigma_m|$ it follows that $\pi_k(\rho) \subseteq \sigma_k$ for all~$k$, therefore $\rho \subseteq \sigma_1 \times \dots \times \sigma_m = \simp(\sigma)$. It follows that $x \in |\simp(\sigma)|$.
\end{proof}

\begin{theorem}
    \label{thm:squashing-HE}
    The squashing map $\xi_P : |\simp(P)| \to |P|$ is a homotopy equivalence.
\end{theorem}

\begin{proof}
    We consider the point fibers of $\xi_P$. Let $y \in |P|$ and take $\sigma \in P$ such that $y \in |\sigma|$. Proposition~\ref{prop:squash-preimage} implies:
    \[
    \xi_P^{-1}(y) = \xi^{-1}(y) \cap |\simp(\sigma)|
    \]
    This is convex because $\xi$ is linear and $|\simp(\sigma)|$ convex; and nonempty by Proposition~\ref{prop:squash-sigma}. Thus $\xi_P^{-1}(y)$ is contractible and locally contractible.
    This proves the result when $I_1, \dots, I_m$ are finite, by the Smale mapping theorem~\ref{thm:smale-v2}.
    
    In general, $\xi_P$ restricts to homotopy equivalences
    \[
    \xymatrix{
    |\simp(P(J_1, \dots, J_m))|
        \ar[r]
    &
    |P(J_1, \dots, J_m)|
    }
    \]
    where
    \[
    P(J_1, \dots, J_m)
    =
    P \cap (\Delta J_1 \times \dots \times \Delta J_m)
    \]
    whenever we have finite subsets $J_1 \subseteq I_1$, \dots, $J_m \subseteq I_m$.
    Let $F \subseteq |\simp(P)|$ and $G \subseteq |P|$ be finite subcomplexes. Each cell, and therefore $F$ and $G$ themselves, involves only finitely many elements of the sets $I_1, \dots, I_m$. Thus we have finite subsets $J_1, \dots, J_m$ such that
    \[
    F \subseteq |\simp(P(J_1, \dots, J_m))|
    \quad \text{and} \quad
    G \subseteq |P(J_1, \dots, J_m)|.
    \]
    The conditions of Proposition~\ref{prop:finite-to-infinite} are satisfied, so $\xi_P$ is a homotopy equivalence.
\end{proof}

\subsection{Functoriality and naturality}
\label{subsec:simp-squash-fun-natural}

The constructions of section~\ref{subsec:simplexification} are categorically well-behaved: simplexification is a functor and the squashing map is a natural transformation.

\begin{proposition}[simplexification is functorial]
\label{prop:simp-functorial}
Let $P$ and $Q$ be prod-complexes on $I_1, \dots, I_m$ and $J_1, \dots, J_m$ and let $f = (f_1, \dots, f_m)$
be a prod-map $P \to Q$. Then the product map
\[
\xymatrix{
I_1 \times \dots \times I_m
    \ar[rr]^-{f_1 \times \dots \times f_m}
&&
J_1 \times \dots \times J_m
}
\]
induces a simplicial map
$
\simp(f) : \simp(P) \to \simp(Q)
$.
This makes $\simp$ a functor from the category of prod-complexes and prod-maps (of fixed order~$m$) to the category of simplicial complexes and simplicial maps.
\end{proposition}

\begin{proof}
Each $\rho \in \simp(P)$ is a subset of some $\sigma_1 \times \dots \times \sigma_m$ where $(\sigma_1, \dots, \sigma_m) \in P$. Then
\[
(f_1 \times \dots \times f_m)
(\sigma_1 \times \dots \times \sigma_m)
=
f_1(\sigma_1) \times \dots \times f_m(\sigma_m)
\]
which is a simplex of $\simp(Q)$ because
\[
(f_1(\sigma_1), \dots, f_m(\sigma_m))
=
f(\sigma_1, \dots, \sigma_m)
\]
is a prod-simplex of~$Q$.
It is readily verified that $\simp(gf) = \simp(g) \circ \simp(f)$ and $\simp(1_P) = 1_{\simp(P)}$ so $\simp$ is a functor.
\end{proof}

\begin{proposition}
\label{prop:epi}
    The functor $\simp$ preserves epimorphisms: if a prod-map $f : P \to Q$ is surjective, then $\simp(f) : \simp(P) \to \simp(Q)$ is surjective.
\end{proposition}

\begin{proof}
    For $\tau \in Q$, take $\sigma \in P$ with $f(\sigma) = \tau$. Then $\simp(f)$ carries $\simp(\sigma)$ onto $\simp(\tau)$, so $\Delta(\simp(\tau))$ is contained in the image of $\simp(f)$. Now take the union over $\tau \in Q$.
\end{proof}

\begin{proposition}[squashing is natural]
    \label{prop:squash-natural}
    Let $P$ and $Q$ be prod-complexes on $I_1, \dots, I_m$ and $J_1, \dots, J_m$ and let
    $
    f = (f_1, \dots, f_m) : P \to Q
    $
    be a prod-map. Then the diagram of spaces
    \begin{equation}
    \label{eq:squash-natural}
    \xymatrix{
    |\simp(P)|
        \ar[rr]^{\simp(f)}
        \ar[d]_{\xi_P}
    &&
    |\simp(Q)|
        \ar[d]^{\xi_Q}
    \\
    |P|
        \ar[rr]^{f}
    &&
    |Q|
    }         
    \end{equation}
    commutes. In other words, the squashing map is a natural transformation.
\end{proposition}

\begin{proof}
The diagram is the restriction of a diagram of vector spaces:
\[
\xymatrix{
\R_{I_1 \times \dots \times I_m}
    \ar[d]_{\xi}
    \ar[rrr]^{\R_{f_1 \times \dots \times f_m}}
&&&
\R_{J_1 \times \dots \times J_m}
    \ar[d]^{\xi}
\\
\R_{I_1} \times \dots \times \R_{I_m}
    \ar[rrr]^{\R_{f_1} \times \dots \times \R_{f_m}}
&&&
\R_{J_1} \times \dots \times \R_{J_m}
}
\]
To see that this commutes, observe that a typical basis element $\basis((i_1,\dots,i_m))$ on the top left is mapped by either path to $(\basis(f_1(i_1)), \dots, \basis(f_m(i_m)))$ on the bottom right.
\end{proof}

\begin{corollary}[squashing respects Dowkerian quotients]
    \label{cor:squash-dowkerian}
    Let $P$ be a prod-complex on $I_1, \dots, I_m$ and let $\psi_k : P \to P/I_k$ denote its $k$-th Dowkerian quotient map.
    Then $\simp(\psi_k)$ is surjective and the diagram of spaces
    \[
    \xymatrix{
    |\simp(P)|
        \ar[rr]^-{|\simp(\psi_k)|}
        \ar[d]_{\xi_P}
    &&
    |\simp(P/I_k)|
        \ar[d]^{\xi_{P/I_k}}
    \\
    |P|
        \ar[rr]^{|\psi_k|}
    &&
    |P/I_k|
    }
    \]
    commutes.
\end{corollary}

\begin{proof}
    Apply Propositions \ref{prop:epi} and~\ref{prop:squash-natural} to $\psi_k$ interpreted according to Remark~\ref{rem:dowkerian-*}.
\end{proof}

\subsection{Multiway Dowker for simplicial complexes}
\label{subsec:multiway-simplicial}

The simplicial version of our main theorem relates a `{cuboid complex}' to its associated {multiway Dowker complexes}. We define these terms now.
Let $R \subseteq I_1 \times \dots \times I_m$ be a multiway relation, and let
\[
    R_{i_k}
    =
    \{
    (i_1, \dots, \widehat{i_k}, \dots, i_m)
    \mid
    (i_1, \dots, i_k, \dots, i_m) \in R
    \}
    \subseteq
    I_1 \times \dots \times \widehat{I_k}
    \times \dots \times I_m
\]
denote the `slice' of~$R$ at $i_k \in I_k$.

\begin{definition}[cuboid complex]
\label{def:cuboid}
    The cuboid complex of~$R$ is the simplicial complex on $I_1 \times \dots \times I_m$ defined
    \begin{align*}
    \cuboid(R)
    &=
    \{
    \rho
    \mid
    \rho_1 \times \dots \times \rho_m \subseteq R \}
    \\
    &=
    \{
    \rho
    \mid
    \text{$\exists\; \sigma_1, \dots, \sigma_m$ such that $\rho \subseteq \sigma_1 \times \dots \times \sigma_m \subseteq R$}
    \}
    \\
    &=
    \simp(\dowker(R))
    \end{align*}
    where $\rho_k = \proj_{I_k}(\rho) \subseteq I_k$.
    When $m=2$, this is the rectangle complex of Brun and Salbu~\cite{Brun_Salbu_2022}. We sometimes write $\rect(R)$ instead of $\cuboid(R)$ as a reminder that $m=2$.
\end{definition}

\begin{definition}[multiway Dowker complex]
\label{def:multiway-Dowker-complex}
    The $k$-th multiway Dowker complex of~$R$ is the simplicial complex on $I_1 \times \dots \times \widehat{I_k} \times \dots \times I_m$ defined as follows:
    \[
    \Dowker_{/I_k}(R)
    =
    \bigcup_{i_k \in I_k} \cuboid(R_{i_k})
    =
    \bigcup_{i_k \in I_k} \simp(\dowker(R_{i_k}))
    =
    \simp(\dowker(R)/I_k)
    \]
\end{definition}

\begin{definition}[simplicial Dowkerian quotient]
\label{def:simplicial-dowkerian-quotient}
    The $k$-th simplicial Dowkerian quotient map is the simplicial map $\varphi_k$ defined by the diagram
    \begin{equation}
    \label{eq:cuboid-dowker-square}
        \xymatrix{
        \cuboid(R)
            \ar[rr]^-{\varphi_k}
            \ar@{=}[d]
        &&
        \Dowker_{/I_k}(R)
            \ar@{=}[d]
        \\
        \simp(\dowker(R))
            \ar[rr]^-{\simp(\psi_k)}
        &&
        \simp(\dowker(R)/I_k)    
        }
    \end{equation}
    and induced by the projection map
    \[
    \xymatrix{
    I_1 \times \dots \times I_m
        \ar[rr]^-{\proj_{/I_k}}
    &&   
    I_1 \times \dots \times \widehat{I_k}
    \times \dots \times I_m
    }
    \]
    on vertices.
    In fact $\Dowker_{/I_k}(R)$ is the smallest simplicial complex to which $\proj_{/I_k}$ defines a simplicial map from $\cuboid(R)$, since $\varphi_k = \simp(\psi_k)$ is surjective (Corollary~\ref{cor:squash-dowkerian}).
\end{definition}

\begin{remark}
    In the classical case of a relation $R \subseteq I \times J$, we have
    \[
    \Dowker_I(R) = \Dowker_{/J}(R) 
    \quad \text{and} \quad
    \Dowker_J(R) = \Dowker_{/I}(R)
    \]
    so we caution the reader to be careful with the subscripts.
    In the case of a ternary relation between $I, J, K$ one may wish to write
    \[
    \Dowker_{JK} = \Dowker_{/I},
    \quad
    \Dowker_{IK} = \Dowker_{/J},
    \quad
    \Dowker_{IJ} = \Dowker_{/K}
    \]
    as we have done in the caption to Figure~\ref{fig:hexagon}.
\end{remark}

The simplicial version of our main theorem now follows quickly.

\begin{theorem}[simplicial multiway Dowker]
\label{thm:multiway-simplicial}
    Let $R \subseteq I_1 \times \dots \times I_k$ be a multiway relation. Then
    \[
    \xymatrix{
    |\cuboid(R)|
    \ar[r]^-{|\varphi_k|}
    &
    |\Dowker_{/I_k}(R)|    
    }
    \]
    is a homotopy equivalence for all~$k$.
\end{theorem}

\begin{proof}
    Corollary~\ref{cor:squash-dowkerian} applied to $P = \dowker(R)$ yields a commutative diagram of spaces
    \[
    \xymatrix{
    |\simp(\dowker(R))|
        \ar[rr]^-{|\simp(\psi_k)|}
        \ar[d]_{\xi}
    &&
    |\simp(\dowker(R)/I_k)|
        \ar[d]^{\xi}
    \\
    |\dowker(R)|
        \ar[rr]^{|\psi_k|}
    &&
    |\dowker(R)/I_k|
    }
    \]
    where the lower three sides are homotopy equivalences.
    Thus $|\simp(\psi_k)| = |\varphi_k|$ is a homotopy equivalence.
\end{proof}

\begin{corollary}
    \label{cor:multiway-square}
    We have a commutative square of homotopy equivalences
    \[
    \xymatrix{
    |\cuboid(R)|
        \ar[rr]^-{|\varphi_k|}
        \ar[d]_{\xi}
    &&
    |\Dowker_{/I_k}(R)|
        \ar[d]^{\xi}
    \\
    |\dowker(R)|
        \ar[rr]^-{|\psi_k|}
    &&
    |\dowker(R)/I_k|
    }
    \]
    combining \eqref{eq:cuboid-dowker-square} with the diagram in the proof of Theorem~\ref{thm:multiway-simplicial}.
    \qed
\end{corollary}

\begin{proposition}[naturality of $\varphi_k$]
\label{prop:multiway-simplicial-functoriality}
Let $R \subseteq I_1 \times \dots \times I_m$ and $S \subseteq J_1 \times \dots \times J_m$. Let
\[
f = (f_1, \dots, f_m)
\quad
\text{where $f_k : I_k \to J_k$}
\]
satisfy $(f_1 \times \dots \times f_m)(R) \subseteq S$. Then we have a commutative diagram
\[
\xymatrix{
    \cuboid(R)
        \ar[rr]^-{\varphi_k}
        \ar[d]_{f_1 \times \dots \times f_m}
    &&
    \Dowker_{/I_k}(R)
        \ar[d]^{f_1 \times \dots \times \widehat{f_k} \times \dots \times f_m}
    \\
    \cuboid(S)
    \ar[rr]^-{\varphi_k}
    &&
    \Dowker_{/J_k}(S)
}
\]
of simplicial maps.
\end{proposition}

\begin{proof}
    The left vertical map is simplicial because of the condition on $f_1 \times \dots \times f_m$. The right vertical map is simplicial because it is the union of the simplicial maps
    \[
    \xymatrix{
    \cuboid(R_{i_k})
        \ar[rrr]^-{f_1 \times \dots \times \widehat{f_k} \times \dots \times f_m}
    &&&
    \cuboid(S_{f_k(i_k)})
    }
    \]
    over $i_k \in I_k$.
    The diagram commutes because the vertices of $\cuboid(R)$ are mapped
    \[
    \xymatrix{
    (i_1, \dots, i_m)
        \ar@{|->}[r]
    &
    (f_1(i_1), \dots, \widehat{f_k(i_k)}, \dots f_m(i_m))
    }
    \]
    by either path through the diagram.
\end{proof}

We combine Theorem \ref{thm:multiway-simplicial} and Proposition~\ref{prop:multiway-simplicial-functoriality} into a single statement.

\begin{theorem}
\label{thm:multiway-simplicial-final}
Let $R \subseteq I_1 \times \dots \times I_m$ be a multiway relation. Then we have a functorially defined diagram
\[
    \xymatrix@C=1.8pc{
    &&& \cuboid(R)
        \ar@{>>}[llld]_-{\varphi_1}
        \ar@{>>}[ld]^(0.4){\varphi_2}
        \ar@{>>}[rrrd]^-{\varphi_m}
    \\
    \Dowker_{/I_1}(R)
    &&
    \Dowker_{/I_2}(R)
    &&
    \dots
    &&
    \Dowker_{/I_m}(R)
    }
    \]
of simplicial maps whose geometric realizations are homotopy equivalences.
\qed
\end{theorem}

\section{Homotopy types of ternary relations}
\label{sec:ternary}

In this closing section, we describe various complexes (section~\ref{subsec:ternary-complexes}) and maps of complexes (section~\ref{subsec:ternary-pairs}) that can be derived, Dowker-style, from a ternary relation.

\subsection{Ternary Dowker complexes}
\label{subsec:ternary-complexes}

Let $R \subseteq I \times J \times K$ be a ternary relation. There are four homotopy types immediately associated to~$R$, namely the homotopy types of its Dowker relational product and its iterated Dowkerian quotients:

\begin{itemize}
    \item
    $|\dowker(R)| \simeq
    |\dowker(R)/I| \simeq
    |\dowker(R)/J| \simeq
    |\dowker(R)/K|$

    \medskip
    \item
    $|\dowker(R)/(J,K)|$

    \medskip
    \item
    $|\dowker(R)/(I,K)|$
    
    \medskip
    \item
    $|\dowker(R)/(I,J)|$
    
\end{itemize}
These are distinct in general.

Next we consider six binary relations derived from $R$. There are three projections:
\begin{alignat*}{7}
&
R_{/I}
&&=
R_{JK}
&&=
\{
(j,k)
&&
\mid
\text{there exists\;} i \in I
&&\text{\;such that\;} (i,j,k) \in R
\}
&&
\; \subseteq \;
J \times K
\\
&
R_{/J}
&&=
R_{IK}
&&=
\{
(i,k)
&&\mid
\text{there exists\;} j \in J
&&\text{\;such that\;} (i,j,k) \in R
\}
&&
\; \subseteq \;
J \times K
\\
&
R_{/K}
&&=
R_{IJ}
&&=
\{
(i,j)
&&
\mid
\text{there exists\;} k \in K
&&\text{\;such that\;} (i,j,k) \in R
\}
&&
\; \subseteq \;
I \times J
\end{alignat*}
And there are three `Cartesian rebracketings':
\begin{alignat*}{5}
&
R_{(JK)}
&&=
\{
(i, (j,k))
&&
\mid
(i,j,k) \in R
\}
&&
\; \subseteq \;
&&
I \times (J \times K)
\\
&
R_{(IK)}
&&=
\{
(j, (i,k))
&&
\mid
(i,j,k) \in R
\}
&&
\; \subseteq \;
&&
J \times (I \times K)
\\
&
R_{(IJ)}
&&=
\{
(k, (i,j))
&&
\mid
(i,j,k) \in R
\}
&&
\; \subseteq \;
&&
K \times (I \times J)
\end{alignat*}
Thus we have ostensibly six additional homotopy types:
\begin{itemize}
    \item
    $|\dowker(R_{JK})|
    \simeq |\Dowker_J(R_{JK})|
    \simeq |\Dowker_K(R_{JK})|
    $

    \medskip
    \item
    $|\dowker(R_{IK})|
    \simeq |\Dowker_I(R_{IK})|
    \simeq |\Dowker_K(R_{IK})|
    $

    \medskip
    \item
    $|\dowker(R_{IJ})|
    \simeq |\Dowker_I(R_{IJ})|
    \simeq |\Dowker_J(R_{IJ})|
    $

 \end{itemize}
and:
 \begin{itemize}
    \item
    $|\dowker(R_{(JK)})|
    \simeq |\Dowker_I(R_{(JK)})|
    \simeq |\Dowker_{J \times K}(R_{(JK)})|
    $

    \medskip
    \item
    $|\dowker(R_{(IK)})|
    \simeq |\Dowker_J(R_{(IK)})|
    \simeq |\Dowker_{I \times K}(R_{(IK)})|
    $

    \medskip
    \item
    $|\dowker(R_{(IJ)})|
    \simeq |\Dowker_K(R_{(IJ)})|
    \simeq |\Dowker_{I \times J}(R_{(IJ)})|
    $

\end{itemize}
 See Figure~\ref{fig:rco-projected} for illustrations of the first three of these six.

\begin{figure}
\centering
\includegraphics[scale=1]{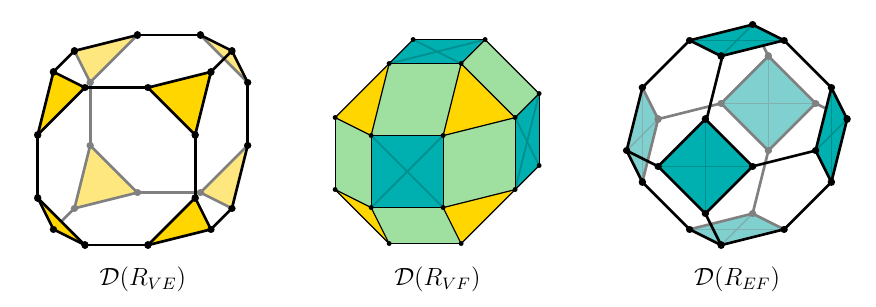}
    \caption{Continuing Figures \ref{fig:rco} and~\ref{fig:rco-iterated}, three additional homotopy types (to complete the set of seven) are realized as the relational products of the three binary relations obtained as projections of the ternary relation~$R$.}
    \label{fig:rco-projected}
\end{figure}

Of the ten homotopy types, only seven are distinct:

\begin{proposition}
    \label{prop:seven-of-ten}
    Let $R \subseteq I \times J \times K$ be a ternary relation. Then:
    \begin{alignat*}{2} 
    &\dowker(R)/(J,K) &&= \Dowker_I(R_{(JK)})
    \\
    &\dowker(R)/(I,K) &&= \Dowker_J(R_{(IK)})
    \\
    &\dowker(R)/(I,J) &&= \Dowker_K(R_{(IJ)})
    \end{alignat*}
\end{proposition}

\begin{proof}
    We prove the first of these (the other two being similar):
    \begin{align*}
    \dowker(R)/(J,K)
    &=
    \{
    \sigma \in \Delta I
    \mid
    \text{there exist $j \in J$, $k \in K$}
    \text{\;such that\;}
    (\sigma, \{j\}, \{k\}) \in \dowker(R)
    \}
    \\
    &=
    \{
    \sigma \in \Delta I
    \mid
    \text{there exist $j \in J$, $k \in K$}
    \text{\;such that\;}
    \sigma \times \{j\} \times \{k\} \subseteq R
    \}
    \\
    &=
    \{
    \sigma \in \Delta I
    \mid
    \text{there exists $(j,k) \in J \times K$}
    \text{\;such that\;}
    \sigma \times \{(j,k)\} \subseteq R_{(JK)}
    \}
    \\
    &=
    \Dowker_I(R_{(JK)})
    \qedhere
    \end{align*}
\end{proof}

\begin{figure}
\centering
\includegraphics[scale=1]{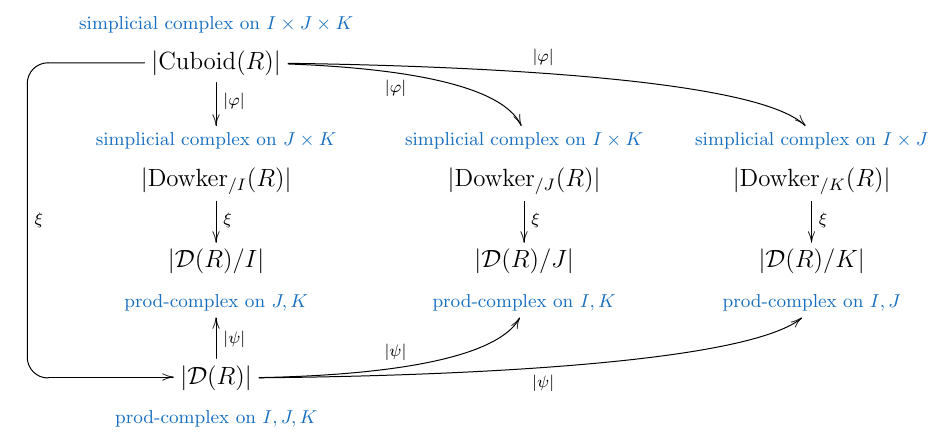}
    \caption{Eight geometric realizations of the homotopy type $|\dowker(R)|$ in a commutative diagram of homotopy equivalences.}
    \label{fig:geometric-representations-1}

\bigskip
\bigskip

\centering
\includegraphics[scale=1]{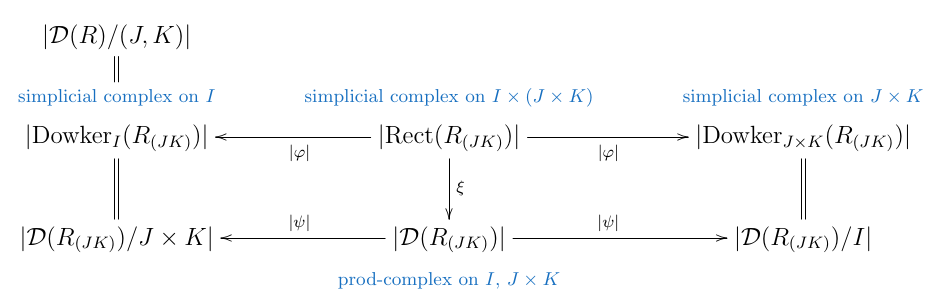}
    \caption{Four geometric realizations (with seven descriptions) of the homotopy type $|\dowker(R_{(JK)})|$ in a commutative diagram of homotopy equivalences.
    }
    \label{fig:geometric-representations-2}

\bigskip
\bigskip

\centering
\includegraphics[scale=1]{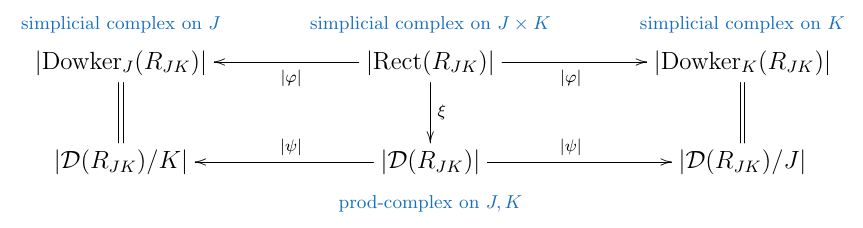}
    \caption{Four geometric realizations (with six descriptions) of the homotopy type $|\dowker(R_{JK})|$ in a commutative diagram of homotopy equivalences.
    }
    \label{fig:geometric-representations-3}
\end{figure}

Each of the seven homotopy types may be geometrically realized in several different ways:
see Figures \ref{fig:geometric-representations-1}, \ref{fig:geometric-representations-2}, \ref{fig:geometric-representations-3}.
These diagrams commute (Corollary~\ref{cor:multiway-square}) so the different realizations are {canonically} homotopy equivalent.
The maps defining the homotopy equivalences are natural transformations 
(Remark~\ref{rem:multiway-naturality}, Proposition~\ref{prop:squash-natural}, Proposition~\ref{prop:multiway-simplicial-functoriality})
so each commutative diagram is functorial in~$R$.\footnote{%
The arrows in Figures \ref{fig:geometric-representations-2} and~\ref{fig:geometric-representations-3} are functorial in~$R$ because the binary relations $R_{(JK)}$ and $R_{JK}$ depend functorially on~$R$: 
given a product map $f \times g \times h : I \times J \times K \to L \times M \times N$ such that $(f \times g \times h)(R) \subseteq S$ it follows that $(f \times (g \times h))(R_{(JK)}) \subseteq S_{(MN)}$ and $(g \times h)(R_{JK}) \subseteq S_{MN}$.}

It follows that persistence modules defined using different realizations of the same homotopy type will be canonically isomorphic. Suppose $A : I \times J \times K \to \R$ is a three-dimensional array, and let $R^t = A^{-1}(-\infty,t]$. Then, for instance, the persistence modules $\mathrm{H}_*(|\dowker(R^t)/(J,K)|)$ and $\mathrm{H}_*(|\dowker(R^t_{(JK)})/I|)$ are isomorphic through the diagram in Figure~\ref{fig:geometric-representations-2}.

If we include the similar forms obtained by reordering $I,J,K$ then Figures \ref{fig:geometric-representations-1}, \ref{fig:geometric-representations-2}, \ref{fig:geometric-representations-3} identify twenty-two different simplicial complexes associated to a ternary relation, organized by homotopy type. For a complementary presentation more suited to finding relationships between these constructions, Figure~\ref{table:constructions-by-vertex-set} catalogues the complexes by their vertex sets.

\begin{remark}
\label{rem:table-tops}
By examining the definitions in the right-hand column of Figure~\ref{table:constructions-by-vertex-set}, the reader may confirm that the first listed entry for each vertex set---$\cuboid(R)$, $\Dowker_{/K}(R)$, $\Dowker_I(R_{JK})$---is a subcomplex of the other entries with the same vertex set. These are the only subcomplex inclusions between simplicial complexes of Figure~\ref{table:constructions-by-vertex-set} that hold generally.
\end{remark}

\begin{figure}
    \bigskip
    \centering
    \begin{tabular}{|l|l|l||l|lll|}
    \hline
    \text{\footnotesize vertex set}
    &
    \text{\footnotesize simplicial complex}
    &
    \text{\footnotesize homotopy type}
    &
    \multicolumn{4}{|l|}{
    \text{\footnotesize simplex membership criterion}}
    \\\hline\hline
    &&&&&&\\[-10pt]
    &
    $\cuboid(R)$
    &
    $\dowker(R)$
    &
    &
    &
    $\rho_I \times \rho_J \times \rho_K$
    &$\subseteq R$
    \\[4pt]
    \multirow{2}[2]{*}{$I \times J \times K$}
    &
    $\rect(R_{(JK)})$
    &
    $\dowker(R)/(J,K)$
    &
    \multirow{2}[2]{*}{$\rho$}
    &
    &
    $\rho_I \times \rho_{JK}$
    &$\subseteq R_{(JK)}$
    \\[4pt]
    &
    $\rect(R_{(IK)})$
    &
    $\dowker(R)/(I,K)$
    &
    &
    &
    $\rho_J \times \rho_{IK}$
    &$\subseteq R_{(IK)}$
    \\[4pt]
    &
    $\rect(R_{(IJ)})$
    &
    $\dowker(R)/(I,J)$
    &
    &
    &
    $\rho_K \times \rho_{IJ}$
    &$\subseteq R_{(IJ)}$
    \\[4pt]
    \hline
    &&&&&&\\[-10pt]
    &
    $\Dowker_{/K}(R)$
    &
    $\dowker(R)$
    &
    &
    $(\exists k)$
    &
    $\zeta_I \times \zeta_J \times \{k\}$
    &
    $\subseteq R$
    \\[4pt]
    $I \times J$
    &
    $\Dowker_{I \times J}(R_{(IJ)})$
    &
    $\dowker(R)/(I,J)$
    &
    $\zeta$
    &
    $(\exists k)$
    &
    $\{k\} \times \zeta$
    &
    $\subseteq R_{(IJ)}$
    \\[4pt]
    &
    $\rect(R_{IJ})$
    &
    $\dowker(R_{IJ})$
    &
    &
    &
    $\zeta_I \times \zeta_J$
    &
    $\subseteq R_{IJ}$
    \\[4pt]
    \hline
    &&&&&&\\[-10pt]
    &
    $\Dowker_I(R_{(JK)})$
    &
    $\dowker(R)/(J,K)$
    &&
    $(\exists j)(\exists k)$
    &
    $\sigma \times \{j\} \times \{k\}$
    &
    $\subseteq R$
    \\[4pt]
    $I$
    &
    $\Dowker_I(R_{IJ})$
    &
    $\dowker(R_{IJ})$
    &
    $\sigma$
    &
    $(\exists j)$
    &
    $\sigma \times \{j\}$
    &
    $\subseteq R_{IJ}$
    \\[4pt]
    &
    $\Dowker_I(R_{IK})$
    &
    $\dowker(R_{IK})$
    &
    &
    $(\exists k)$
    &
    $\sigma \times \{k\}$
    &
    $\subseteq R_{IK}$
    \\[4pt]
    \hline
    \end{tabular}
    \caption{Simplicial complexes associated to a ternary relation.
    Subscripts on a simplex denote projection: $\rho_I = \proj_I(\rho)$, $\rho_{IJ} = \proj_{I \times J}(\rho)$ and so on.
    The first entry in each block is a subcomplex of the entries below it.
    }
    \label{table:constructions-by-vertex-set}
\end{figure}

\subsection{Ternary Dowker pairs}
\label{subsec:ternary-pairs}

In this section we identify twelve natural transformations between the seven homotopy types constructed in section~\ref{subsec:ternary-complexes}: a set of three, a set of six, and three more by composition. See Figures \ref{fig:ternary-natural},~\ref{fig:ternary-natural-rco}.

\begin{figure}
\[
\xymatrix@R=4pc@C=4pc{
&
|\dowker(R)|
    \ar[dl]
    \ar[d]
    \ar[dr]
\\
|\dowker(R)/(J,K)|
    \ar[dr]
    \ar[d]
&
|\dowker(R)/(I,K)|
    \ar[dl]
    \ar[dr]
&
|\dowker(R)/(I,J)|
    \ar[d]
    \ar[dl]
\\
|\dowker(R_{IJ})|
&
|\dowker(R_{IK})|
&
|\dowker(R_{JK})|
}
\]
    \caption{Commutative diagram of natural transformations between the seven homotopy types associated to a ternary relation~$R$.
    }
    \label{fig:ternary-natural}

\bigskip
\bigskip

\centering
\includegraphics[scale=1]{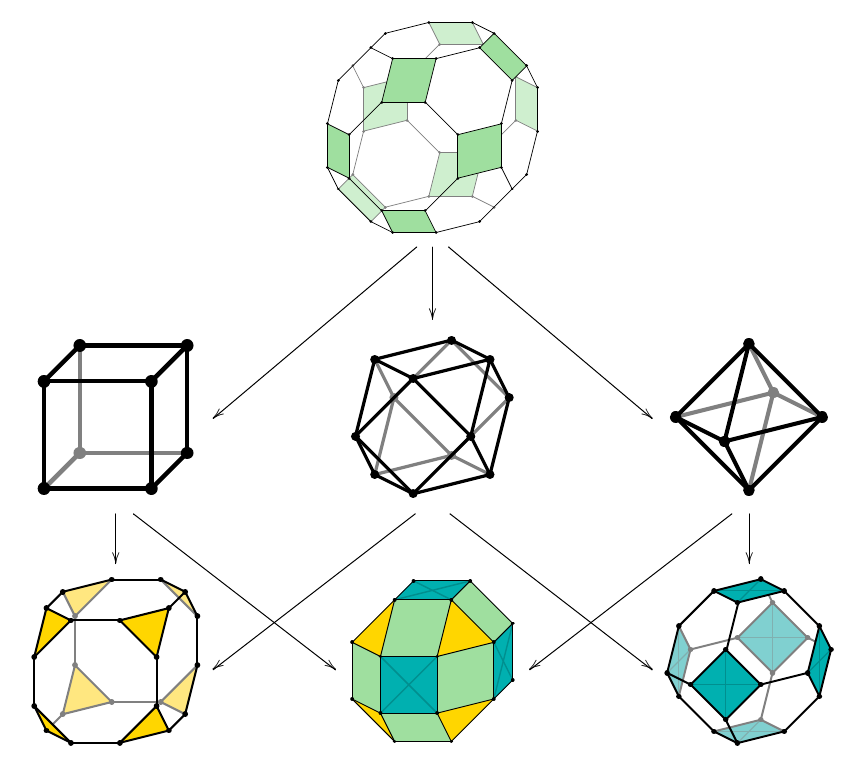}

    \caption{The diagram of natural transformations in Figure~\ref{fig:ternary-natural} applied to the running example from Figures \ref{fig:rco-iterated-all}, \ref{fig:rco}, \ref{fig:rco-iterated}, \ref{fig:rco-projected}.}
    \label{fig:ternary-natural-rco}
\end{figure}

Each natural transformation is realized as a map between complexes representing the two homotopy types involved. The homotopy type of the mapping cone---the `homotopy cofiber'---depends only on the homotopy type of the map, which we are free to choose wisely.
In each case there is a \emph{subcomplex realization}: an inclusion map of a subcomplex in a larger simplicial complex on the same vertex set.
Thus we can compute the reduced homology of the mapping cone as the relative homology of the pair of complexes.

We now construct Figure~\ref{fig:ternary-natural}. The top three maps are the iterated Dowkerian quotients. Each pair of these maps form two sides of a parallelogram that commutes up to homotopy. 
The following diagram
constructs the left parallelogram, the other two being similar.
\begin{equation}
\label{eqn:construct-ternary-nt}
\raisetag{-5ex}
\xymatrix@R=2.5pc@C=3pc{
&
|\dowker(R)|
    \ar[dl]
    \ar[dr]
    \ar[d]_{|\psi_K|}^{\simeq}
\\
|\dowker(R)/(J,K)|
    \ar[d]_{A}
    \ar@{.>}[dr]^{A'}
&
|\dowker(R)/K|
    \ar[d]_{P}
    \ar[l]_-{|\psi_J|}
    \ar[r]^-{|\psi_I|}
&
|\dowker(R)/(I,K)|
    \ar[d]^{B}
    \ar@{.>}[dl]_{B'}
\\
|\dowker(R_{IJ})/J|
&
|\dowker(R_{IJ})|
    \ar[l]^-{|\psi_J|}_-{\simeq}
    \ar[r]_-{|\psi_I|}^-{\simeq}
&
|\dowker(R_{IJ})/I|
}
\end{equation}
The $\psi$ maps are Dowkerian quotients, three of which are homotopy equivalences.
The map~$P$ is a geometric prod-complex inclusion.
This exists because
$\sigma \times \tau \times \{k\} \subseteq R$ implies $\sigma \times \tau \subseteq R_{IJ}$.
The maps $A,B$ are inclusions that follow from Remark~\ref{rem:table-tops}:
\begin{alignat*}{7}
&\dowker(R)/(J,K)
&& \xymatrix{\ar@{=}[r] &}
&& \Dowker_I(R_{(JK)})
&& \xymatrix{\ar[r] &}
&& \Dowker_I(R_{IJ})
&& \xymatrix{\ar@{=}[r] &}
&&\dowker(R_{IJ})/J
\\
&\dowker(R)/(I,K)
&& \xymatrix{\ar@{=}[r] &}
&& \Dowker_J(R_{(IK)})
&& \xymatrix{\ar[r] &}
&& \Dowker_J(R_{IJ})
&& \xymatrix{\ar@{=}[r] &}
&&\dowker(R_{IJ})/I
\end{alignat*}
Their existence, easily verified directly, stems from the status of $\psi_{J}$, $\psi_{I}$ as quotient maps. Finally, in the homotopy category, we can and must take $A' \simeq \alpha A$ and $B' \simeq \beta B$, where $\alpha, \beta$ are homotopy inverses to $|\psi_J|, |\psi_I|$ in the bottom row.

\begin{remark}
    The diagram {without} $A',B'$ is functorial within the category of spaces. The diagram including $A',B'$ is functorial only within the homotopy category, because the homotopy inverses $\alpha, \beta$ are specified only up to homotopy.  
\end{remark}

Next we realize the natural transformations in Figure~\ref{fig:ternary-natural} as simplicial maps, preferably as subcomplex inclusions. Up to permutations of $I,J,K$ there are three types to consider.

\begin{proposition}
    \label{prop:nt-realizations}
    {\bf (i)}
    The map
    $|\dowker(R)| \longrightarrow |\dowker(R_{IJ})|$
    is realized by these simplicial maps:
    \begin{alignat*}{4}
        &
        \cuboid(R)
        &&
        \xymatrix{\ar[r]&}
        \rect(R_{IJ})
        &&
        \text{\small\rm (induced by $I \times J \times K 
        \longrightarrow
        I \times J$)}
        \\
        &
        \Dowker_{/K}(R)
        &&
        \xymatrix{\ar[r]&}
        \rect(R_{IJ})
        &\qquad&
        \text{\small\rm (subcomplex inclusion, Remark~\ref{rem:table-tops})}
    \end{alignat*}

    \noindent
    {\bf (ii)}
    The map $|\dowker(R)| \longrightarrow |\dowker(R)/(J,K)|$ is realized by these simplicial maps:
    \begin{alignat*}{4}
        &
        \cuboid(R)
        &&
        \xymatrix{\ar[r]&}
        \rect(R_{(JK)})
        &&
        \text{\small\rm (subcomplex inclusion, Remark~\ref{rem:table-tops})}
        \\
        &
        \cuboid(R)
        &&
        \xymatrix{\ar[r]&}
        \Dowker_I(R_{(JK)})
        &&
        \text{\small\rm (induced by $I \times J \times K 
        \longrightarrow
        I$)}
        \\
        &
        \cuboid(R)
        &&
        \xymatrix{\ar[r]&}
        \Dowker_{J \times K}(R_{(JK)})
        &&
        \text{\small\rm (induced by $I \times J \times K 
        \longrightarrow
        J \times K$)}
        \\
        &
        \Dowker_{/I}(R)
        &&
        \xymatrix{\ar[r]&}
        \Dowker_{J \times K}(R_{(JK)})
        &\qquad&
        \text{\small\rm (subcomplex inclusion, Remark~\ref{rem:table-tops})}
    \end{alignat*}

    \noindent
    {\bf (iii)}
    The map $|\dowker(R)/(J,K)| \longrightarrow |\dowker(R_{IJ})|$ is realized by this simplicial map:
    \begin{alignat*}{4}
        &
        \Dowker_I(R_{(JK)})
        &&
        \xymatrix{\ar[r]&}
        & 
        \Dowker_I(R_{IJ})
        &\qquad&
        \text{\small\rm (subcomplex inclusion, Remark~\ref{rem:table-tops})}
    \end{alignat*}
\end{proposition}

\begin{proof}[Proof of Proposition~\ref{prop:nt-realizations}]
    Part {\bf (i)} follows from the simplexification diagram
    \[
    \xymatrix{
    |\dowker(R)|
        \ar[r]^-{\simeq}
    &
    |\dowker(R)/K|
        \ar[r]^{}
    &
    |\dowker(R_{IJ})|
    \\
    |\cuboid(R)|
        \ar[r]_-{\simeq}
        \ar[u]^{\simeq}_{\xi}
    &
    |\Dowker_{/K}(R)|
        \ar[r]
        \ar[u]^{\simeq}_{\xi}
    &
    |\rect(R_{IJ})|
        \ar[u]^{\simeq}_{\xi}
    }
    \]
    whose upper row is the central column of
    diagram~\eqref{eqn:construct-ternary-nt}

    Part {\bf (ii)} follows from the simplexification diagram
    \[
    \xymatrix{
    |\dowker(R)|
        \ar[rrrr]^(0.45){|\psi_J||\psi_K| \,=\, |\psi_K||\psi_J|}
    &&&&
    |\dowker(R)/(J,K)|
    \\
    |\cuboid(R)|
        \ar[u]^{\simeq}_{\xi}
        \ar[rrrr]
            ^(0.45){|\varphi_J||\varphi_K| \,=\, |\varphi_K||\varphi_J|}
    &&&&
    |\Dowker_I(R_{(JK)})|
        \ar[u]^{\simeq}_{\xi}
    }
    \]
    of the iterated Dowkerian quotient map, along with the diagram of simplicial maps
    \[
    \xymatrix{
    \cuboid(R)
        \ar[r]
        \ar[d]_{\simeq}
    &
    \rect(R_{(JK)})
        \ar[r]^-{\simeq}
        \ar[d]_{\simeq}
    &
    \Dowker_I(R_{(JK)})
    \\
    \Dowker_{/I}(R)
        \ar[r]
    &
    \Dowker_{J \times K}(R_{JK})
    }
    \]
    three of whose maps are Dowker homotopy equivalences.

    Part {\bf (iii)} follows from the simplexification diagram
    \[
    \xymatrix{
    |\dowker(R)/(J,K)|
        \ar[r]
    &
    |\dowker(R_{IJ})/J|
    &
    |\dowker(R_{IJ})|
        \ar[l]_-{\simeq}
    \\
    |\Dowker_I(R_{(JK)})|
        \ar[u]^{\simeq}_{\xi}
        \ar[r]
    &
    |\Dowker_I(R_{IJ})|
        \ar[u]^{\simeq}_{\xi}
    &
    |\rect(R_{IJ})|
        \ar[u]^{\simeq}_{\xi}
        \ar[l]^-{\simeq}    
    }
    \]
    whose upper row is the lower left corner of
    diagram~\eqref{eqn:construct-ternary-nt}
\end{proof}

\begin{remark}
    \label{rem:no-triple}
    Whereas each individual natural transformation may be realized by a subcomplex inclusion, the factorization
    \[
    |\dowker(R)| \longrightarrow |\dowker(R)/(J,K)| \longrightarrow |\dowker(R_{IJ})|
    \]
    is not realized as a composite of two of these inclusions. The following diagram perhaps comes closest to such a realization:
    \[
    \xymatrix{
    \cuboid(R)
        \ar[rr]^-{\subseteq}
        \ar[d]^{\simeq}
    &&
    \rect(R_{(JK)})
        \ar[r]_-{\simeq}
    &
    \Dowker_I(R_{(JK)})
        \ar[rr]^-{\subseteq}
    &&
    \Dowker_I(R_{IJ})
    \\
    \Dowker_{/K}(R)
        \ar[rrrrr]_-{\subseteq}
    &&&&&
    \rect(R_{IJ})
        \ar[u]^{\simeq}
    }
    \]
    The three connecting maps are Dowker homotopy equivalences.
\end{remark}

This completes our construction of the diagram of natural transformations in Figure~\ref{fig:ternary-natural}. Each of the twelve maps yields, as an invariant, the homotopy type of its mapping cone (its homotopy cofiber). In total, we have nineteen homotopy types associated functorially to a ternary relation.
For the twelve natural transformations, we may calculate the relative homology of their subcomplex realizations:
\begin{alignat*}{4}
&
\Hgr_* \big(|\dowker(R)| \longrightarrow |\dowker(R_{IJ})| \big)
&\quad&
=
&\quad&
\Hgr_*\big(
    \rect(R_{IJ}), \Dowker_{/K}(R)
\big)
\\
&
\Hgr_*\big(
|\dowker(R)| \to |\dowker(R)/(J,K)|
\big)
&\quad&
=
&\quad&
\begin{cases}
    \Hgr_*\big( \rect(R_{(JK)}),\cuboid(R) \big)
    \\
    \Hgr_*\big(
        \Dowker_{J \times K}(R_{(JK)}), \Dowker_{/I}(R)
    \big)
\end{cases}
\\
&
\Hgr_*\big(
|\dowker(R)/(J,K)| \to |\dowker(R_{IJ})|
\big)
&\quad&
=
&\quad&
\Hgr_*\big(
    \Dowker_I(R_{IJ}), \Dowker_I(R_{(JK)})
\big)
\end{alignat*}
Each pair yields a long exact sequence;  likewise each of the six composites of the form
\[
|\dowker(R)|
\longrightarrow |\dowker(R)/(I,J)|
\longrightarrow |\dowker(R_{IK})|
\]
yields the long exact sequence of a triple.
Finally, all of this can be done in parametrized settings. In particular, we get persistence diagrams in relative homology for each of the twelve natural transformations, alongside persistence diagrams in absolute homology for the seven homotopy types. The interpretation of these diagrams we leave as an open question.

As a closing exercise, we invite the reader to construct the different geometric realizations of the seven homotopy types of the running example
(Figures \ref{fig:rco-iterated-all}, \ref{fig:rco}, \ref{fig:rco-iterated}, \ref{fig:rco-projected}, \ref{fig:ternary-natural-rco})
and determine the homotopy cofibers of the twelve natural transformations. For example, the homotopy cofiber of $|\dowker(R)| \to |\dowker(R_{VF})|$ is a bouquet of fourteen 2-spheres.

\section*{Acknowledgements}

VdS thanks Robert Ghrist for introducing him to Dowker's theorem many years ago, and dedicates section~\ref{subsec:dowkerian-quotients} to his first undergraduate supervision partner. 
CG was partially supported by the Air Force Office of Scientific Research under award number FA9550-21-1-0266.
VI was partially supported by the NSF Next Generation Networks for Neuroscience Program (award 2014217).
MR was partially supported by the Office of Naval Research (ONR) under contract number N00014-22-1-2659.
RS was partially supported by the NSF grant DMS-185470.
NS completed part of this work while appointed at the University of Delaware and supported by the Air Force Office of Scientific Research under
award number FA9550-21-1-0266.
All of the authors express their gratitude to Chad Giusti, Gregory Henselman-Petrusek, and Lori Ziegelmeier, for organizing the workshop at which this project began; and to the American Institute of Mathematics for hosting the workshop.

\bibliographystyle{plainurl}
\bibliography{references}

\begin{thebibliography}{10}

\bibitem{Begle_1950}
Edward~G. Begle.
\newblock The {V}ietoris mapping theorem for bicompact spaces.
\newblock {\em Annals of Mathematics}, 51(3):534--543, 1950.

\bibitem{Bjorner:1995}
Anders Bj{\"o}rner.
\newblock Topological methods.
\newblock In R.~L. Graham, M.~Grotschel, and L.~Lovasz, editors, {\em Handbook
  of Combinatorics}, pages 1819--1872. MIT Press, Cambridge, MA, USA, 1995.

\bibitem{Brown_2006}
Ronald Brown.
\newblock {\em {Topology and Groupoids}}.
\newblock BookSurge Publishing, 2006.

\bibitem{Brun_Grinberg_2024}
Morten Brun and Darij Grinberg.
\newblock The {D}owker theorem via discrete {M}orse theory, 2024.
\newblock \href {http://arxiv.org/abs/2407.15454} {\path{arXiv:2407.15454}}.

\bibitem{Brun_Salbu_2022}
Morten Brun and Lars~M. Salbu.
\newblock The rectangle complex of a relation.
\newblock {\em Mediterranean Journal of Mathematics}, 20(1), 2022.

\bibitem{Chowdhury_Memoli}
Samir Chowdhury and Facundo M\'{e}moli.
\newblock A functorial {D}owker theorem and persistent homology of asymmetric
  networks.
\newblock {\em Journal of Applied and Computational Topology}, 2:115--175, 10
  2018.

\bibitem{Dowker}
C.~H. Dowker.
\newblock Homology groups of relations.
\newblock {\em Annals of Mathematics}, 56(1):84--95, 1952.

\bibitem{Dowker_1952b}
C.~H. Dowker.
\newblock Topology of metric complexes.
\newblock {\em American Journal of Mathematics}, 74(3):555--577, 1952.

\bibitem{Hatcher_2002}
Allen Hatcher.
\newblock {\em Algebraic Topology}.
\newblock Cambridge University Press, Cambridge, 2002.

\bibitem{Kozlov_2008}
Dmitry Kozlov.
\newblock {\em Combinatorial Algebraic Topology}.
\newblock Springer Berlin Heidelberg, 2008.

\bibitem{Smale_1957}
Stephen Smale.
\newblock A {V}ietoris mapping theorem for homotopy.
\newblock {\em Proceedings of the American Mathematical Society},
  8(3):604--610, 1957.

\bibitem{Vietoris_1927}
L.~Vietoris.
\newblock {\"U}ber den h{\"o}heren {Z}usammenhang kompakter {R}{\"a}ume und
  eine {K}lasse von zusammenhangstreuen {A}bbildungen.
\newblock {\em Mathematische Annalen}, 97:454--472, 1927.

\bibitem{Virk_2021}
\v{Z}iga Virk.
\newblock {R}ips complexes as nerves and a functorial {D}owker-nerve diagram.
\newblock {\em Mediterranean Journal of Mathematics}, 18(2):58, 2021.

\bibitem{Whitehead_1949}
J.~H.~C. Whitehead.
\newblock Combinatorial homotopy {I}.
\newblock {\em Bulletin of the American Mathematical Society}, 55(5):213--245,
  1949.

\bibitem{Yoon_2024}
Iris H.~R. Yoon.
\newblock Dowker duality, profunctors, and spectral sequences, 2024.
\newblock \href {http://arxiv.org/abs/2408.13136} {\path{arXiv:2408.13136}}.

\end{thebibliography}

\end{document}